\documentclass[3p,sort&compress]{article}
\usepackage[a4paper, total={6.5in, 8.6in}]{geometry}

\usepackage{amssymb}
\usepackage{amsthm}
\newtheorem{theorem}{Theorem}
\newtheorem{proposition}[theorem]{Proposition}
\newtheorem{lemma}[theorem]{Lemma}
\newtheorem{corollary}[theorem]{Corollary}
\theoremstyle{definition}
\newtheorem{definition}{Definition}

\usepackage{mathtools,mathrsfs}
\DeclarePairedDelimiter{\group}{(}{)}
\DeclarePairedDelimiter{\sqgroup}{[}{]}
\DeclarePairedDelimiter{\set}{\{}{\}}

\DeclarePairedDelimiter{\abs}{\vert}{\vert}

\newcommand{\pths}{\Omega}
\newcommand{\event}{A}

\newcommand{\interval}{\sqgroup{\underline{p},\overline{p}}}
\newcommand{\intervalpq}{\sqgroup{p,q}}
\newcommand{\intervals}{\mathscr{I}}

\newcommand{\sits}{\mathbb{S}}

\newcommand{\frcstsystem}{\varphi}
\newcommand{\frcstsystempq}{\frcstsystem_{p,q}^\varpi}

\newcommand{\lfrcstsystem}{\underline{\frcstsystem}}
\newcommand{\ufrcstsystem}{\overline{\frcstsystem}}
\newcommand{\frcstsystems}{\Phi}

\newcommand{\processgeneral}{F}
\newcommand{\process}{F}
\newcommand{\diffprocess}{\adddelta F}

\newcommand{\selection}{S}

\newcommand{\selections}{\mathscr{S}}
\newcommand{\selectionsrec}{\mathscr{S}}

\newcommand{\multprocess}{D}
\newcommand{\mint}[1][\multprocess]{#1^{\raisebox{0.5pt}{$\scriptstyle \circledcirc$}}}

\newcommand{\supermartin}{M}
\newcommand{\realgrowth}{\tau}

\newcommand{\submartin}{M}
\newcommand{\submartins}[1][\frcstsystem]{\underline{\mathbb{M}}(#1)}
\newcommand{\supermartins}[1][\frcstsystem]{\overline{\mathbb{M}}(#1)}

\newcommand{\testsupermartins}[1][\frcstsystem]{\overline{\mathbb{T}}(#1)}

\newcommand{\testsupermartinsrandomopen}[1]{\overline{\mathbb{T}}_\random(#1)}

\newcommand{\test}{T}

\newcommand{\naturals}{\mathbb{N}}
\newcommand{\naturalswithzero}{{\mathbb{N}_0}}

\newcommand{\reals}{\mathbb{R}}
\newcommand{\posreals}{\mathbb{R}_{>0}}
\newcommand{\nonnegreals}{\mathbb{R}_{\geq0}}

\newcommand{\rationals}{\mathbb{Q}}

\newcommand{\ex}{E}
\newcommand{\lex}{\underline{\ex}}

\newcommand{\uex}{\overline{\ex}}

\newcommand{\pr}{P}
\newcommand{\lpr}{\underline{\pr}}
\newcommand{\upr}{\overline{\pr}}

\newcommand{\lexglobal}[1][\frcstsystem]{\lex^{#1}}

\newcommand{\uexglobal}[1][\frcstsystem]{\uex^{#1}}

\newcommand{\lprglobal}[1][\frcstsystem]{\lpr^{#1}}
\newcommand{\uprglobal}[1][\frcstsystem]{\upr^{#1}}

\newcommand{\gambleinterval}[1][]{\sqgroup{\underline{\gamble},\overline{\gamble}}}

\newcommand{\init}{\square}
\newcommand{\pth}{\omega}
\newcommand{\varpth}{\varpi}
\newcommand{\pthat}[1]{\pth_{#1}}
\newcommand{\pthatn}{\pth_{n}}
\newcommand{\pthatnplus}{\pth_{n+1}}

\newcommand{\pthatkplus}{\pth_{k+1}}
\newcommand{\pthto}[1]{\pth_{1:#1}}
\newcommand{\pthton}{\pth_{1:n}}

\newcommand{\varpthatnplus}{\varpi_{n+1}}

\newcommand{\varpthatkplus}{\varpi_{k+1}}

\newcommand{\pthtok}{\pth_{1:k}}

\newcommand{\sit}{s}

\newcommand{\xval}[1][]{x_{#1}}
\newcommand{\xvaltolong}[1][n]{\xval[1],\dots,\xval[#1]}
\newcommand{\xvalto}[1]{\xval[1:#1]}
\newcommand{\xvalton}{\xval[1:n]}

\newcommand{\xvaltonn}{\xval[1:n+1]}
\newcommand{\xvaltok}{\xval[1:k]}
\newcommand{\xvaltom}{\xval[1:m]}

\newcommand{\xvalatkplus}{\xval[k+1]}
\newcommand{\xvalatn}{\xval[n]}
\newcommand{\xvalatnplus}{\xval[n+1]}

\newcommand{\gamble}{f}
\newcommand{\gambleglobal}{f}
\newcommand{\gamblesglobal}{\mathscr{L}(\pths)}

\newcommand{\gambles}{\mathscr{L}(\set{0,1})}

\newcommand{\comp}{computable}

\newcommand{\lscomp}{lower semicomputable}
\newcommand{\uscomp}{upper semicomputable}
\newcommand{\compy}{computability}

\newcommand{\ind}[1]{\mathbb{I}_{#1}}

\newcommand{\adddelta}{\Delta}

\newcommand{\cset}[3][]{\set[#1]{#2\colon#3}}

\newcommand{\random}{\textnormal{R}}
\newcommand{\ch}{\textnormal{CH}}
\newcommand{\wch}{\textnormal{wCH}}
\newcommand{\co}{\textnormal{C}}
\newcommand{\s}{\textnormal{S}}
\newcommand{\ml}{\textnormal{ML}}
\newcommand{\wml}{\textnormal{wML}}

\newcommand{\realordering}{\tau}

\renewcommand{\cset}[3][]{\set[#1]{#2\colon#3}}

\newcommand{\mltests}[1]{\overline{\mathbb{T}}_\ml(#1)}
\newcommand{\weakmltests}[1]{\overline{\mathbb{T}}_\wml(#1)}
\newcommand{\comptests}[1]{\overline{\mathbb{T}}_\co(#1)}
\newcommand{\schnorrtests}[1]{\overline{\mathbb{T}}_\s(#1)}
\newcommand{\randomtests}[1]{\overline{\mathbb{T}}_\random(#1)}
\usepackage{enumerate}
\usepackage{enumitem}
\usepackage{tikz}
\usetikzlibrary{positioning}
\usepackage{adjustbox}
\usepackage{float}
\usepackage{setspace}
\usepackage{stackengine}
\usepackage{graphicx}
\usepackage{nicefrac}
\usepackage{aligned-overset}
\usepackage{comment}
\usepackage{hyperref}
\usepackage{sectsty}
\sectionfont{\fontsize{13}{0}\selectfont}
\allowdisplaybreaks

\newcommand{\revised}[1]{{\color{black}#1}}
\newcommand{\revisedit}[1]{{\color{black}#1}}

\begin{document}

\begin{center}

\Large{On the (dis)similarities between stationary imprecise and non-stationary precise uncertainty models in algorithmic randomness}
\\ \vspace{10pt}
\normalsize{Floris Persiau, Jasper De Bock, Gert de Cooman}
\\
\textit{\footnotesize{Ghent University, Foundations Lab for imprecise probabilities, Technologiepark--Zwijnaarde 125, 9052 Zwijnaarde, Belgium}}
\end{center}

\author{Floris Persiau \and Jasper De Bock \and Gert de Cooman}
\date{}


\small \vspace{10pt}
\begin{minipage}{0.85\textwidth}
\textbf{Abstract}\vspace{5pt}\\ 
\noindent
The field of algorithmic randomness studies, amongst other things, what it means for infinite binary sequences to be random for some given uncertainty model.
Classically, martingale-theoretic notions of randomness involve precise uncertainty models, and it is only recently that imprecision has been introduced into this context.
As a consequence, the investigation into how imprecision alters our view on martingale-theoretic random sequences has only just begun.
In this contribution, where we allow for non-{\comp} uncertainty models, we establish a close and surprising connection between precise and imprecise uncertainty models in this randomness context.
In particular, we show that there are stationary imprecise models and non-{\comp} non-stationary precise models that have the exact same set of random sequences.
We also give a preliminary discussion of the possible implications of our result for a statistics based on imprecise probabilities, and shed some light on the practical (ir)relevance of both imprecise and non-{\comp} precise uncertainty models in that context.\vspace{5pt}
\end{minipage}
\normalsize

\section{Introduction}
What does it mean for an infinite binary sequence, such as \(011001100\cdots\), to be random?
This is a highly non-trivial question that has led to numerous investigations.
First of all, it is important to realise that randomness is typically defined with respect to some uncertainty model.
So, our opening question only makes sense once such a model has been specified.
Uncertainty models can be stationary or non-stationary, {\comp} or non-{\comp}, as well as precise or imprecise \revisedit{\cite{CoomanBock2017,CoomanBock2021,CoomanBock2021V2,Vovk2010,Levin1973,Levin1976,Bienvenu2011,DowneyHirschfeldt2010}}.
It is between the non-{\comp} non-stationary precise and the stationary imprecise uncertainty models that we will reveal a remarkably close connection: we will show that for every stationary imprecise model there are non-{\comp} non-stationary precise models that have the exact same set of random sequences.

The earliest notions of randomness only considered precise probability models that assign a probability~\(p\in\sqgroup{0,1}\) to the outcome~\(1\).
For instance, in 1919 Von Mises suggested considering an infinite binary sequence to be random for~\(p\) if the relative frequencies of ones along the sequence, and along all infinite subsequences selected by selection rules, converge to~\(p\) \cite{Mises1919}.
In 1939, Wald proved that for any~\(p\in\group{0,1}\), such random sequences do exist if we restrict our attention to a countable set of selection rules \cite{ambosspies2000,wald1937}.
This result, however, left open which countable set to consider.
In 1940, based on Wald's work, Church suggested adopting the countable set of {\comp} selection rules \cite{ambosspies2000,church1940}; an infinite binary sequence is then Church random for~\(p\) if the relative frequency of ones along every computably selectable infinite subsequence converges to~\(p\), where `computably selectable' essentially means that there is some finite algorithm that decides which elements to keep and which to discard.

However, there are infinite sequences that satisfy this requirement, but for which the running frequency of ones along the sequence converges to \(p\) from below.
Such sequences disobey the law of the iterated logarithm.
For this reason, Jean Ville criticised this type of randomness definition, and argued that besides the law of large numbers, a random sequence also ought to satisfy other statistical laws \cite{ambosspies2000}.
Arguments of this kind led to the development of many other randomness notions.
Some of the most well-known and well-studied amongst these are Martin-Löf randomness, {\comp} randomness and Schnorr randomness \cite{MartinLof1966,Schnorr1971book}.
The reason for this is twofold: they have an intuitive interpretation and they can be defined in several equivalent ways \cite{ChallengeOfChance2016,DowneyHirschfeldt2010}.
From a measure-theoretic point of view, for example, an infinite binary sequence is random for a {\comp} real~\(p \in\sqgroup{0,1}\) if it passes all implementable statistical tests that are associated with~\(p\).
On the other hand, if we adopt the martingale-theoretic approach, then a sequence is random for a {\comp}~\(p\) if there is no implementable betting strategy for getting arbitrarily rich along this sequence without borrowing, where the bets that are allowed are determined by~\(p\), where the betting strategy specifies the possible accumulated capital in the betting game, and where the meaning of `implementable' depends on the notion of randomness at hand; for instance, in the case of {\comp} randomness, `implementable' means that there is a finite algorithm that yields the strategy.

There is more to randomness, however, than the simple case of a single {\comp} probability~\(p\).
As mentioned above, more general uncertainty models, such as non-stationary precise ones or imprecise ones, can also be used to define notions of randomness \revisedit{\cite{Vovk2010,Bienvenu2011,Levin1973,Levin1976}}; these models need not necessarily be computable either.
\revised{In a measure-theoretic context, most of the classical approaches impose computability \cite{MartinLof1966,Schnorr1971book,DowneyHirschfeldt2010}, but non-computable uncertainty models have been studied too. One approach, for example, is to consider the implementable statistical tests that are associated with a non-computable measure, but that do not access it as a resource; this notion of randomness is known as \emph{Hippocratic} or \emph{Blind randomness} \cite{Bjorn2010,Bjorn2014,Bienvenu2011}.
Another measure-theoretic notion of randomness that allows for non-computable uncertainty models was put forward by Levin in 1973 and is nowadays known as \emph{uniform randomness} \revisedit{\cite{Levin1973,Levin1976,Bienvenu2011}.}
This notion of uniform randomness also allows for imprecision by considering so-called `effectively compact classes of probability measures'; in particular, there are tests such that a sequence passes this test if and only if it is uniformly random with respect to some probability measure in the considered class.
}

In the context of this paper, however, we focus on the martingale-theoretic approach to randomness, for which imprecise-probabilistic uncertainty models have been considered only recently.
To be more precise, in the past few years, De Cooman and De Bock put forward a martingale-theoretic approach that associates (weak) Martin-Löf, {\comp} and Schnorr randomness with imprecise rather than precise probability models \cite{CoomanBock2017,CoomanBock2021,CoomanBock2021V2}. \revised{These imprecise models take the form of so-called forecasting systems, and need not be computable.}
This recent work still leaves room for many open questions on how allowing for imprecision (and letting go of {\comp} uncertainty models) changes our understanding of random sequences.
In the present paper, we try and contribute to this understanding by proving a remarkable relation between randomness for precise and imprecise probability models.
In particular, for every non-vanishing closed interval~\(I\subseteq\sqgroup{0,1}\) and each of the above-mentioned four martingale-theoretic notions of randomness, we will show that there is a non-stationary precise but then necessarily non-{\comp} uncertainty model for which the set of random paths is the same as for \(I\).
So, for the results in this paper, allowing for non-{\comp} uncertainty models is of crucial importance.
We leave aside whether using non-{\comp} uncertainty models in algorithmic randomness is defensible on philosophical or theoretical grounds; we simply let go of the classical computability restriction on uncertainty models, and investigate what happens if we do so.
Nevertheless, in our conclusions, we do argue why {\comp} uncertainty models are to be favoured on practical grounds.

Our contribution is structured as follows.
\revised{Sections~\ref{sec:single}--\ref{sec:comp} provide a short overview of relevant earlier work.}
We start by introducing interval forecasts and the associated coherent upper expectations in Section~\ref{sec:single}, where we also explain how their interpretation leads to a convex cone of gambles that a subject is willing to offer.
Section~\ref{sec:multiple} explains how to bet on a single variable in a way that agrees with an interval forecast, and extends this idea to a betting game/protocol on an infinite sequence of variables by defining betting strategies---which are basically supermartingales---that again agree with such interval forecasts, and that avoid borrowing.
After clarifying in Section~\ref{sec:comp} when such betting strategies are implementable, we present in Section~\ref{sec:borrow} the imprecise-probabilistic martingale-theoretic notions of (weak) Martin-Löf, {\comp} and Schnorr randomness introduced in earlier work \cite{CoomanBock2021,sum2020persiau}, and we discuss some of their properties.
At this point, in Section~\ref{sec:difference}, we are ready to formulate our central claim/result: for each of the above-mentioned martingale-theoretic notions of randomness, an infinite sequence is random for an interval forecast if and only if it is random for some specific related non-{\comp} non-stationary precise uncertainty model.
\revised{We complement this result with a reflection on what it tells us about allowing for imprecision and non-computability in a martingale-theoretic approach to algorithmic randomness.
Moreover, we explain that our result has implications for the interpretation of imprecise randomness, and that it finds a reflection in results for the measure-theoretic approach.}
In the subsequent three sections, we work towards the proof of our main claim.
In Section~\ref{sec:almostsure}, we use betting strategies to introduce global upper expectations and define almost sure events.
In Section~\ref{sec:freq} we introduce, inspired by Wald's work, a very general imprecise-probabilistic frequentist notion of randomness in terms of countable sets of selection processes, and highlight the particular imprecise-probabilistic frequentist notions of (weak) Church randomness that we introduced elsewhere \cite{CoomanBock2021,floris2021ecsqaru}.
In Section~\ref{sec:proof} then, we employ all this mathematical machinery in the `construction' of the specific non-stationary precise uncertainty models that we can use to prove our main theorem.
We conclude this paper with a discussion in Section~\ref{sec:stat} on the possible implications of our main result for a prospective statistics based on imprecise probabilities.

\section{Local uncertainty models: interval forecasts and gambles}\label{sec:single}
We begin our discussion by considering a single variable~\(X\) that may assume some value in the binary outcome space~\(\set{0,1}\).
To describe a subject's uncertainty about the unknown value of~\(X\), we use a closed interval~\(I\subseteq\sqgroup{0,1}\).
We collect all such closed intervals in the set~\(\intervals\) and call them \emph{interval forecasts}.
One way to interpret an interval forecast~\(I\in\intervals\) is to regard its elements~\(p\in I\) as possible values for the probability that \(X\) equals~\(1\), or equivalently, for the expectation of, or fair price for, the uncertain reward~\(X\).
In this paper, however, where betting will play a central role, we prefer to adopt a different interpretation.
We interpret the lower and upper bound of~\(I\) as a subject's largest acceptable buying and smallest acceptable selling price, respectively,\footnote{Traditionally, in a so-called imprecise probabilities context, the lower and upper bound of~\(I\) are interpreted as a subject's supremum acceptable buying and infimum acceptable selling price for the uncertain pay-off~\(X\). However, as was proved in \cite[Appendix A]{CoomanBock2021}, our four imprecise-probabilistic martingale-theoretic notions of randomness are the same under both interpretations. We adopt the former interpretation, because it simplifies some of our proofs.} for the uncertain pay-off~\(X\in\set{0,1}\), expressed in units of some linear utility scale.

Consequently, if \(I=\interval\), our subject is willing to accept the uncertain pay-off~\(X-p\) for any buying price~\(p\leq\underline{p}\), and is willing to accept the uncertain pay-off~\(q-X\) for any selling price~\(q\geq\overline{p}\).
Due to the assumed linearity of the utility scale, this implies that he is willing to \emph{accept} the uncertain pay-off~\(\alpha(X-p)+\beta(q-X)\) for any~\(p\leq\underline{p}\), \(q\geq\overline{p}\) and~\(\alpha,\beta\geq0\).
From the perspective of an opponent who bets against our subject, this means that our subject is willing to \emph{offer} her any uncertain reward of the form
\begin{equation}\label{eq:gambles:offered}
\alpha(p-X)+\beta(X-q)
\text{ with~\(p\leq\underline{p}\), \(q\geq\overline{p}\) and~\(\alpha,\beta\geq0\)}.
\end{equation}
To manipulate these uncertain rewards mathematically, it will be convenient to identify them with maps on~\(\set{0,1}\), whose value in~\(x\in\set{0,1}\) is obtained by replacing \(X\) with~\(x\).
The reward~\(X\), for example, then corresponds to the identity map on~\(\set{0,1}\).
We will call any such map~\(\gamble\colon\set{0,1}\to\reals\) from the binary sample space to the \emph{real} numbers a \emph{gamble}, and we denote the set of all gambles by~\(\gambles\).
Since \(\abs{\set{0,1}}=2\), gambles can be drawn in a two-dimensional space.
This allows us to visualise the set of all gambles of the type~\eqref{eq:gambles:offered}, which are offered by our subject to an opponent as a result of the commitments implicit in his specifying the interval forecast~\(I=[\underline{p},\overline{p}]\).
It is clear from Figure~\ref{fig:cone} that this set is a convex cone that includes the third quadrant.

\begin{figure}[ht]
\centering
{\caption{Let \(\underline{p}\coloneqq\nicefrac{1}{4}\) and~\(\overline{p}\coloneqq\nicefrac{3}{4}\). The green region depicts all gambles~\(\gamble\in\gambles\) that correspond to an uncertain reward~\(\alpha(p-X)+\beta(X-q)\), with~\(p\leq\underline{p}\), \(q\geq\overline{p}\) and~\(\alpha,\beta\geq0\).}\label{fig:cone}}
{
\begin{tikzpicture}[scale=1.65]
\draw[step=.5cm,gray,very thin] (-2.25,-2.25) grid (1.25,1.25);
\fill[opacity=0.5,green!60!black] (0,0) -- (-5/2*9/10,5/6*9/10) -- (-5/2*9/10,-5/2*9/10) -- (5/6*9/10,-5/2*9/10) --cycle;
\draw[->] (-2.25,0) -- (1.25,0) node[below right] {\small\(f(1)\)};
\draw[->] (0,-2.25) -- (0,1.25) node[above left] {\small\(f(0)\)};
\foreach \x in {-2,-1,1}
    \draw (\x,2pt) -- (\x,-2pt)
	node[anchor=north] {\small\(\x\)};
\foreach \y in {-2,-1}
    \draw (2pt,\y) -- (-2pt,\y) 
    node[anchor=east] {\small\(\y\)};
\foreach \y in {1}
    \draw (-2pt,\y) -- (2pt,\y)
    node[anchor=west] {\small\(\y\)};
\node[anchor=north east] at (0, 0)  {\small\(0\)};
\draw[very thick,green!60!black] (0,0) -- (-5/2*9/10,5/6*9/10);
\draw[very thick,green!60!black] (0,0) -- (5/6*9/10,-5/2*9/10);
\filldraw[green!60!black] (0,0) circle (1pt);
\draw (7/4-13/4,2/4) .. controls (7/4-13/4+0.1,1/2+0.2) and (9/4-13/4-0.1,3/4-0.2) .. (9/4-13/4,3/4) node [above,fill=white,rounded corners] {\(\cset{\alpha(\underline{p}-X)}{\alpha\geq0}\)};
\draw (2/4,7/4-13/4) .. controls (2/4+0.1,7/4+0.2-13/4) and (4/4-0.1,8/4-0.2-13/4) .. (4/4,8/4-13/4) node [above right=0 cm and -1cm,fill=white,rounded corners] {\(\cset{\beta(X-\overline{p})}{\beta\geq0}\)};
\end{tikzpicture}
}
\end{figure}

It will be useful to have an analytical condition that, for a subject with interval forecast~\(I\in\intervals\), characterises the gambles he is willing to offer to an opponent.
To this end, we introduce upper (and lower) expectation operators.
When \(I=p\in\reals\), i.e., when \(I\) reduces to a single number, we consider the \emph{linear expectation}~\(\ex_p\) defined by
\begin{equation}
\ex_p(\gamble)\coloneqq p\gamble(1)+(1-p)\gamble(0)\text{ for all }\gamble\in\gambles\label{eq:ex}.
\end{equation}
This is a most informative---or least conservative---model for a subject's uncertainty.
When \smash{\(I=\interval \notin \reals\)}, we consider the \emph{upper expectation}~\(\uex_I\) defined by
\begin{equation}\label{eq:upperex}
\uex_I(\gamble)
\coloneqq\max_{p\in I}\ex_p(f)
=\max\set{\ex_{\underline{p}}(\gamble),\ex_{\overline{p}}(\gamble)}
\text{ for all } \gamble\in\gambles.
\end{equation}
As a closely related operator, we consider the \emph{lower expectation}~\(\lex_I\colon\gambles\to\reals\) defined by
\begin{equation}\label{eq:lowerex}
\lex_I(\gamble)
\coloneqq\min_{p\in I}\ex_p(f)
=\min\set{\ex_{\underline{p}}(\gamble),\ex_{\overline{p}}(\gamble)}
\text{ for all~\(\gamble\in\gambles\).}
\end{equation}
It is clear that lower and upper expectations are related to each other through the following conjugacy relationship: \(\uex_I(\gamble)=-\lex_I(-\gamble)\) for all~\(\gamble\in\gambles\).

It is a matter of straightforward verification that the upper expectation~\(\uex_I\) satisfies the following so-called \emph{coherence} properties \cite{ItIP2014}.
\begin{proposition}
Consider any interval forecast~\(I\in\intervals\).
Then for all gambles~\(\gamble,g\in\gambles\), and all~\(\mu\in\reals\) and~\(\lambda\in\nonnegreals\):\footnote{\(\nonnegreals\) denotes the set of non-negative real numbers, whereas \(\posreals\) denotes the set of positive real numbers.}
\begin{enumerate}[label=\upshape{C}\arabic*.,ref=\upshape{C}\arabic*,leftmargin=*,itemsep=0pt]
\item\label{axiom:coherence:bounds} \(\min \gamble\leq\uex_I(\gamble)\leq\max\gamble\)\hfill{\upshape[boundedness]}
\item\label{axiom:coherence:homogeneity} \(\uex_I(\lambda\gamble)=\lambda\uex_I(\gamble)\)\hfill{\upshape[non-negative homogeneity]}
\item\label{axiom:coherence:subsupadditivity} \(\uex_I(\gamble+g)\leq\uex_I(\gamble)+\uex_I(g)\)\hfill{\upshape[subadditivity]}
\item\label{axiom:coherence:constantadditivity} \(\uex_I(\gamble+\mu)=\uex_I(\gamble)+\mu\)\hfill{\upshape[constant additivity]}
\item if \(\gamble\leq g\) then \(\uex_I(\gamble)\leq\uex_I(g)\) \hfill{\upshape[monotonicity]}\label{axiom:coherence:increasingness}
\end{enumerate}
\end{proposition}

These coherence properties \ref{axiom:coherence:bounds}--\ref{axiom:coherence:constantadditivity} allow us to show fairly directly that the upper expectation~\(\uex_I\) indeed characterises the gambles that are offered by our subject.

\begin{proposition}\label{prop:from:I:to:E:andback}
Consider any gamble~\(\gamble\in\gambles\) and any interval forecast~\(I=\interval\in\intervals\).
Then \(\uex_I(\gamble)\leq 0\) if and only if there are \(p\leq\underline{p}\), \(q \geq\overline{p}\) and~\(\alpha,\beta\geq0\) such that \(\gamble=\alpha(p-X)+\beta(X-q)\).
\end{proposition}

\begin{proof}

To prove the direct implication, assume that \(\uex_I(\gamble)\leq 0\).
We will consider three cases: \(\gamble(1)>\gamble(0)\), \(\gamble(1)<\gamble(0)\) and~\(\gamble(1)=\gamble(0)\).
If \(\gamble(1)>\gamble(0)\), then \(\gamble\) can always be written as \(\gamble=\beta(X-q)\), with~\(\beta=\gamble(1)-\gamble(0)\in\posreals\) and~\(q=\nicefrac{\gamble(0)}{(\gamble(0)-\gamble(1))}\in\reals\).
Since \(\beta>0\), it now suffices to prove that \(q\geq \overline{p}\).
To do so, observe that by Equation~(\ref{eq:upperex}) it holds that \smash{\(\uex_I(X)=\overline{p}\)}.
Consequently, it follows from \ref{axiom:coherence:homogeneity} and \ref{axiom:coherence:constantadditivity} that
\begin{align*}
0
\geq
\uex_I(\gamble)
=\uex_I(\beta(X-q))
=\beta (\uex_I(X)-q)
=\beta (\overline{p}-q),
\end{align*}
and hence, since \(\beta>0\), that \(q\geq \overline{p}\).
If \(\gamble(1)<\gamble(0)\), then \(\gamble\) can always be written as \(\gamble=\alpha(p-X)\), with~\(\alpha=\gamble(0)-\gamble(1)\in\posreals\) and~\(p=\nicefrac{\gamble(0)}{(\gamble(0)-\gamble(1))}\in\reals\).
Since \(\alpha>0\), it now suffices to prove that \(p\leq\underline{p}\).
To do so, observe that by Equation~(\ref{eq:lowerex}) it holds that \smash{\(\lex_I(X)=\underline{p}\)}.
Consequently, it follows from conjugacy, \ref{axiom:coherence:homogeneity} and \ref{axiom:coherence:constantadditivity} that
\begin{align*}
0
\geq\uex_I(\gamble)
=-\lex_I(-\gamble)
=-\lex_I(\alpha(X-p))
=-\alpha\lex_I(X-p)
=-\alpha(\lex_I(X)-p)
=\alpha(p-\lex_I(X))
=\alpha(p-\underline{p}),
\end{align*}
and hence, since \(\alpha>0\), that \(p\leq\underline{p}\).
If \(\gamble(1)=\gamble(0)\), then \(\gamble(1)=\gamble(0)=\uex_I(\gamble)\leq 0\) by \ref{axiom:coherence:bounds}, and hence, \(\gamble\) can always be written as \(\gamble=\uex_I(\gamble)=\alpha(0-X)+\beta(X-1)\), with~\(\alpha=\beta=-\uex_I(\gamble)\geq 0\), \(0\leq\underline{p}\) and~\(1\geq\overline{p}\).

To prove the converse implication, assume that \(\gamble=\alpha(p-X)+\beta(X-q)\), with~\(p\leq\underline{p}\), \(q\geq \overline{p}\) and~\(\alpha,\beta\geq0\).
From \ref{axiom:coherence:homogeneity}-\ref{axiom:coherence:constantadditivity}, conjugacy and by recalling that \smash{\(\uex_I(X)=\overline{p}\)} and \smash{\(\lex_I(X)=\underline{p}\)}, it then immediately follows that
\begin{align*}
\uex_I(\gamble)
&=\uex_I(\alpha(p-X)+\beta(X-q))
\leq\uex_I(\alpha(p-X))+\uex_I(\beta(X-q))
=\alpha\uex_I(p-X)+\beta\uex_I(X-q)\\
&=\alpha(p+\uex_I(-X))+\beta(\uex_I(X)-q)
=\alpha(p-\lex_I(X))+\beta(\uex_I(X)-q)
=\alpha(p-\underline{p})+\beta(\overline{p}-q)
\leq 0,
\end{align*}
which completes the proof.
\end{proof}

\section{Forecasting systems and betting strategies}\label{sec:multiple}
We can better describe the correspondence between a subject's interval forecasts, which specify his beliefs about the unknown outcomes of binary variables, and actual outcome sequences, by taking him up on a specific betting game.

We first introduce the betting game for a single binary variable~\(X\).
There are three players involved: \emph{Forecaster} (who will take up our subject's part), \emph{Sceptic} (who is his opponent) and \emph{Reality}.
Forecaster initiates the game by providing an interval forecast~\(I\subseteq\sqgroup{0,1}\), which describes, as we explained in the previous section, his beliefs about---and betting commitments related to---the uncertain outcome~\(X\in\set{0,1}\).
Next, Sceptic, being Forecaster's opponent, is allowed to pick any gamble~\(\gamble\in\gambles\) that Forecaster is willing to offer, in the specific sense that \(\uex_I(\gamble)\leq 0\).
This leads to an uncertain (possibly negative) gain~\(\gamble(X)\) for Sceptic and~\(-\gamble(X)\) for Forecaster.
Finally, Reality reveals the outcome~\(x\in\set{0,1}\), which leads to an actual (possibly negative) gain~\(\gamble(x)\) for Sceptic and~\(-\gamble(x)\) for Forecaster.

In order to extend these ideas to an infinite betting game involving a sequence of successively revealed binary variables~\(X_1\), \dots, \(X_n\), \dots, we require a bit more terminology.

In what follows, \(\naturals\) denotes the set of natural numbers, and \(\naturalswithzero\coloneqq\naturals\cup\set{0}\) denotes the set of non-negative integers.
An infinite outcome sequence~\((\xvaltolong,\dots)\) is called a \emph{path} and is also denoted by~\(\pth\).
All such paths are collected in the set~\(\pths\coloneqq\set{0,1}^\naturals\), and for every path~\(\pth=(\xvaltolong,\dots)\in\pths\), we let \(\pthton\coloneqq(\xvaltolong)\) and~\(\pthatnplus\coloneqq\xvalatnplus\) for all~\(n\in\naturalswithzero\).
A finite outcome sequence~\(\xvalton\coloneqq (\xvaltolong)\in\set{0,1}^n\) is called a \emph{situation} and is also denoted by~\(\sit\), with length~\(\abs{\sit}\coloneqq n\), for any~\(n\in\naturalswithzero\).
All situations are collected in the set~\(\sits\coloneqq\bigcup_{n\in\naturalswithzero}\set{0,1}^n\).
For any~\(\sit=(\xvaltolong)\in\sits\) and~\(x\in\set{0,1}\), we write \(\sit x\) as a shorthand notation for~\((\xvaltolong,x)\).
We call the empty sequence~\(\init\coloneqq\xvalto{0}=()\) the \emph{initial situation}.
Note that for every path~\(\pth\in\pths\), we have that \(\pthto{0}=\init\).

It will also be useful to be able to deal with objects that depend on the situations.
Formally, we define a \emph{process}~\(\processgeneral\) as a map on the set~\(\sits\) of all situations.
In particular, a \emph{real} process~\(\process\colon\sits\to\reals\) is a map from situations to real numbers.
A real process~\(\process\) is called \emph{non-negative} if \(\process(\sit)\geq0\) for all~\(\sit\in\sits\); it is called \emph{positive} if \(\process(\sit)>0\) for all~\(\sit\in\sits\).
We call a non-negative real process~\(\process\) a \emph{test process} if additionally \(\process(\init)=1\).
A zero-one valued process~\(\selection\)---with~\(\selection(\sit)\in\set{0,1}\) for all~\(\sit\in\sits\)---is called a \emph{selection process}.
If a process~\(\processgeneral\) depends only on the situations~\(\sit\in\sits\) through their length~\(\abs{\sit}\), we call it \emph{temporal}, and then also write \(\processgeneral(n)\) instead of~\(\processgeneral(\sit)\) for all~\(n\in\naturalswithzero\) and~\(\sit\in\sits\) with~\(n=\abs{\sit}\).

Forecaster's part in the game now consists in providing an interval-valued process.
That is, he provides an interval forecast~\(I_{\sit}\in\intervals\) for every finite outcome sequence~\(\sit\in\sits\), in order to describe his beliefs about the binary variable~\(X_{\abs{\sit}+1}\) after observing the outcome~\(\sit\) for the previous \(\abs{\sit}\) variables.
The interval-valued process that summarises these interval forecasts, we call a forecasting system.

\begin{definition}[Forecasting system]
A \emph{forecasting system} is a map~\(\frcstsystem\colon\sits\to\intervals\) that associates with every situation~\(\sit\in\sits\) an interval forecast~\(\frcstsystem(\sit)\in\intervals\).
A forecasting system~\(\frcstsystem\) is called \emph{precise} if \(\frcstsystem(\sit)\in\reals\) for all~\(\sit\in\sits\).
We denote the set~\(\smash{\intervals^{\sits}}\) of all forecasting systems by~\(\frcstsystems\).
\end{definition}

With every forecasting system~\(\frcstsystem\in\frcstsystems\), we associate two real processes~\(\lfrcstsystem\) and~\(\ufrcstsystem\) defined by~\(\lfrcstsystem(\sit)\coloneqq\min\frcstsystem(\sit)\) and~\(\ufrcstsystem(\sit)\coloneqq\max\frcstsystem(\sit)\) for all~\(\sit\in\sits\).
Clearly, a forecasting system~\(\frcstsystem\in\frcstsystems\) is precise if and only if~\(\lfrcstsystem(\sit)=\ufrcstsystem(\sit)\) for all~\(\sit\in\sits\).
If a forecasting system~\(\frcstsystem\in\frcstsystems\) uses the very same interval forecast~\(I\in\intervals\) in all situations, i.e., \(\frcstsystem(\sit)=I\) for all~\(\sit\in\sits\), then we call it \emph{stationary}, and we simplify the notation by writing \(I\) instead of~\(\frcstsystem\).
If for two forecasting systems \(\frcstsystem,\frcstsystem^* \in\frcstsystems\) it holds that \(\frcstsystem(s)\subseteq\frcstsystem^*(s)\) for all~\(s\in\sits\), then we say that \(\frcstsystem^*\) is less informative---or more conservative---than \(\frcstsystem\), and denote this by \(\frcstsystem\subseteq\frcstsystem^*\).
\revised{In this case, if \(\frcstsystem\) is precise, then we also say that \(\frcstsystem\) is \emph{compatible} with \(\frcstsystem^*\), and denote this by \(\frcstsystem\in\frcstsystem^*\). It is also worth noting that, through the Kolmogorov extension theorem, every precise forecasting system uniquely determines a corresponding probability measure on the algebra (of sets of paths) generated by the situations. In measure-theoretic notions of randomness, one will typically focus on this measure instead of its forecasting system. In our game-theoretic approach, however, the forecasting system takes the center stage.}

Once Forecaster has specified a forecasting system~\(\frcstsystem\in\frcstsystems\), Sceptic is allowed to adopt any betting strategy that, for every situation~\(\sit\in\sits\), selects one of the gambles that Forecaster is bound to offer by his specification of the interval forecast~\(\frcstsystem(\sit)\in\intervals\), i.e., some gamble~\(\gamble_{\sit}\in\gambles\) for which \(\uex_{\frcstsystem(\sit)}(\gamble_{\sit})\leq0\).
Afterwards, Reality reveals the successive outcomes~\(X_n=\xvalatn\) at each successive \emph{time instant}~\(n\in\naturals\), leading to the sequence~\(\pth=(\xvaltolong,\dots)\).
At every time instant \(n\), after Reality has revealed the outcome~\(\xvalatn\), Sceptic uses the gamble~\(\gamble_{\xvalton}\) that corresponds to her betting strategy.
Next, Reality reveals the subsequent outcome~\(X_{n+1}=\xval[n+1]\in\set{0,1}\) and the reward~\(\gamble_{\xvalton}(\xval[n+1])\) goes to Sceptic.
Moreover, we will prohibit Sceptic from borrowing.

In order to formalise these betting strategies for Sceptic, we define a \emph{gamble process} as a map from situations to gambles.
In particular, we associate with every real process~\(\process\) a \emph{process difference}~\(\diffprocess\colon\sits\to\gambles\), which is the gamble process that maps any \(\sit\in\sits\) to the gamble~\(\diffprocess(\sit)\coloneqq\process(\sit\,\bullet)-\process(\sit)\), where we use \(\process(\sit\,\bullet)\) to denote the gamble on~\(\set{0,1}\) whose value, for any \(x\in\set{0,1}\), is given by \(\process(\sit x)\).
Note that \(\process(\xvalton)=\process(\init)+\sum_{k=0}^{n-1}\adddelta\process(\xvaltok)(\xvalatkplus)\) for all~\(\xvalton\in\sits\), with~\(n\in\naturalswithzero\).

Given a forecasting system~\(\frcstsystem\in\frcstsystems\), we call a real process~\(\supermartin\) a \emph{supermartingale} for~\(\frcstsystem\) if \smash{\(\uex_{\frcstsystem(\sit)}(\adddelta\supermartin(\sit))\leq 0\)} for all~\(\sit\in\sits\).
A real process~\(\submartin\) is called a \emph{submartingale} for~\(\frcstsystem\) if \(-\submartin\) is a supermartingale for~\(\frcstsystem\), meaning that \(\lex_{\frcstsystem(\sit)}(\adddelta\submartin(\sit))\geq 0\) for all~\(\sit\in\sits\).
All supermartingales and submartingales for~\(\frcstsystem\) are respectively collected in the sets~\smash{\(\supermartins\)} and \smash{\(\submartins\)}.

Supermartingales correspond to Sceptic's allowed betting strategies.
Indeed, assume that Forecaster adopts the forecasting system~\(\frcstsystem\in\frcstsystems\), consider a time instant \(n\in\naturalswithzero\), and consider the situation where Reality has revealed a finite outcome sequence~\(\xvalton\in\sits\).
A supermartingale~\(\supermartin\) for~\(\frcstsystem\) then specifies a gamble~\(\adddelta\supermartin(\xvalton)\in\gambles\) that Sceptic is allowed to pick.
If she does, and Reality reveals the outcome~\(\xvalatnplus\in\set{0,1}\), the (possibly negative) amount~\(\adddelta\supermartin(\xvalton)(\xvalatnplus)\) goes to Sceptic and her total capital becomes
\begin{align*}\label{eq:additive:gambles}
\supermartin(\xvaltonn)
=\supermartin(\xvalton)+\adddelta\supermartin(\xvalton)(\xvalatnplus)
=\supermartin(\init)+\sum_{k=0}^{n}\adddelta\supermartin(\xvaltok)(\xvalatkplus),
\end{align*}
with~\(\supermartin(\init)\) her initial capital.
By focusing on non-negative supermartingales, we additionally prevent Sceptic from borrowing.

As an important special case, we consider \emph{test supermartingales} \(\test \colon\sits\to\reals\) for~\(\frcstsystem\).
These are non-negative supermartingales for~\(\frcstsystem\) for which \(\test(\init)\coloneqq1\).
We collect all test supermartingales for~\(\frcstsystem\) in the set~\(\testsupermartins\).
In one of our notions of randomness, we will adopt a particular way of defining such test supermartingales by focusing on multiplicative rather than additive betting strategies.
For this reason, we introduce the notion of a \emph{multiplier process}, which is a non-negative gamble process.
With every such multiplier process~\(\multprocess\), we associate a test process~\(\mint\), defined by the `initial condition'~\(\mint(\init)\coloneqq 1\) and, for all~\(\sit\in\sits\) and~\(x\in\set{0,1}\), by the recursion equation~\(\mint(\sit x)\coloneqq\mint(\sit) \multprocess(\sit)(x)\), and we say that \(\mint\) is \emph{generated by}~\(\multprocess\).
In particular, we say that a test supermartingale~\(\test\in\testsupermartins\) is \emph{generated by a multiplier process} if there is some non-negative gamble process~\(\multprocess_\test\) such that \(\test(\xvalton)=\prod_{k=0}^{n-1}\multprocess_\test(\xvaltok)(\xvalatkplus)\) for all~\(\xvalton\in\sits\), with~\(n\in\naturalswithzero\).

\section{Computable forecasting systems and implementable betting strategies}\label{sec:comp}
Sceptic will not be allowed to adopt just any non-negative supermartingale as a betting strategy.
We will also require that it should be implementable \cite{ambosspies2000,bienvenu2009}.
Loosely speaking, this means that for each of her betting strategies, there is some finite description that specifies how to approximate it (to arbitrary precision). 
In order to be able to formalise when a supermartingale is implementable, we turn to computability theory.

As a basic building block, this theory considers \emph{recursive} natural functions \(\phi\colon\naturalswithzero\to\naturalswithzero\), which are maps that can be computed by a Turing machine \cite{Pour-ElRichards2016}.
By the Church--Turing thesis, this is equivalent to the existence of a finite algorithm that, given the input \(n\in\naturalswithzero\), outputs \(\phi(n)\in\naturalswithzero\).
For example, since a path~\(\pth\in\pths\) is a function from \(\naturals\) to~\(\set{0,1}\), we call it recursive if there is some finite algorithm that, given the input \(n\in\naturals\), outputs \(\pthatn\in\set{0,1}\).
\revisedit{Instead of~\(\naturalswithzero\), we will also consider functions whose domain equals~\(\naturals\), \(\sits\), \(\sits\times\set{0,1}\), \(\sits\times\naturalswithzero\) or any other countably infinite set~\(\mathscr{D}\) whose elements can be encoded by the natural numbers.
A function $\phi\colon \mathscr{D}\to\naturalswithzero$ is then called \emph{recursive} if there is a Turing machine that, when given the natural-valued encoding of an element $d\in\mathscr{D},$\footnote{\revisedit{The choice of encoding is not important provided different encodings can be translated into each other. The latter means that we can algorithmically decide whether a natural number is an encoding of an object and, if this is the case, that we can find an encoding of the same object with respect to the other encoding \cite{Alexander2017book}.}} outputs $\phi(d)$.
Also in this more general case, by the Church--Turing thesis, this is equivalent to the existence of a finite algorithm that, given the input \(d\in\mathcal{D}\), outputs \(\phi(d)\in\naturalswithzero\). In fact, throughout this paper and its proofs, it is this alternative, more intuitive, characterisation that we will be using.} 
For example, we will call a selection process~\(\selection\) recursive if there is a finite algorithm that, given the input \(\sit\in\sits\), outputs the binary digit \(\selection(\sit)\in\set{0,1}\).

More generally, for any countable domain~\(\mathscr{D}\) \revisedit{that can be encoded by the natural numbers}, a \emph{rational} map~\(q\colon\mathscr{D}\to\rationals\) is called \emph{recursive} if there are three recursive natural maps \(a,b,c\colon\mathscr{D}\to\naturalswithzero\) such that
\begin{equation*}
b(d) \neq 0\text{ and } q(d)=\smash{(-1)^{c(d)}\frac{a(d)}{b(d)}}\text{ for all } d\in\mathscr{D}.
\end{equation*}
Since a finite number of finite algorithms can always be combined into one finite algorithm \cite{MichaelSipser2006}, this is equivalent to the existence of a finite algorithm that outputs \(q(d)\) when given \(d\in\mathscr{D}\) as input.
In particular, a rational process~\(\process\colon\sits\to\rationals\) is called recursive if there is a recursive rational map~\(q\colon\sits\to\rationals\) such that \(\process(\sit)=q(\sit)\) for all~\(\sit\in\sits\).

Computability theory not only considers recursive objects, but also uses them to introduce \emph{\lscomp} ones.
The simplest case is that of a {\lscomp} real number: a real number~\(x\in\reals\) is called {\lscomp} if there is a recursive rational map~\(q\colon\naturalswithzero\to\rationals\) such that \(q(n)\leq q(n+1)\) for all~\(n\in\naturalswithzero\), and~\(x=\lim_{m\to\infty}q(m)\).
More generally, for any countable domain~\(\mathscr{D}\) \revisedit{that can be encoded by the natural numbers}, a real map~\(r\colon\mathscr{D}\to\reals\) is called {\lscomp} if there is a recursive rational map~\(q\colon\mathscr{D}\times\naturalswithzero\to\rationals\) such that \(q(d,n)\leq q(d,n+1)\) and~\(r(d)=\lim_{m\to\infty}q(d,m)\) for all~\(d\in\mathscr{D}\) and~\(n\in\naturalswithzero\).
A real map~\(r\) is thus {\lscomp} if there is some finite algorithm that, for every element \(d\in\mathscr{D}\) of the domain, can provide an increasing sequence~\((q(d,n))_{n\in\naturalswithzero}\) of rational numbers that approaches the real number~\(r(d)\) from below---but without necessarily knowing how good the lower bounds~\(q(d,n)\) are.
In particular, a real process~\(\process\) is called {\lscomp} if there is a recursive rational map~\(q\colon\sits\times\naturalswithzero\to\rationals\) such that \(q(\sit,n)\leq q(\sit,n+1)\) and~\(\process(\sit)=\lim_{m\to\infty}q(\sit,m)\) for all~\(\sit\in\sits\) and~\(n\in\naturalswithzero\).
Similarly, a real multiplier process~\(\multprocess\) is called {\lscomp} if there is a recursive rational map~\(q\colon\sits\times\set{0,1}\times\naturalswithzero\to\rationals\) such that \(q(\sit,x,n)\leq q(\sit,x,n+1)\) and~\(\multprocess(\sit)(x)=\lim_{m\to\infty}q(\sit,x,m)\) for all~\(\sit\in\sits\), \(x\in\set{0,1}\) and~\(n\in\naturalswithzero\).
Two types of {\lscomp} (gamble) processes that we will make frequent use of, are {\lscomp} test supermartingales and {\lscomp} multiplier processes that generate test supermartingales.

We call a real map~\(r\colon\mathscr{D}\to\reals\) \emph{\uscomp} if \(-r\) is {\lscomp}.
If a real map~\(r\) is both lower and {\uscomp}, then we call it \emph{\comp}.
Equivalently \cite{CoomanBock2017}, a real map~\(r\colon\mathscr{D}\to\reals\) is {\comp} if and only if there is a recursive rational map~\(q\colon\mathscr{D}\times\naturalswithzero\to\rationals\) such that
\begin{equation*}
\abs{r(d)-q(d,n)}<2^{-n}\text{ for all } d\in\mathscr{D}\text{ and } n\in\naturalswithzero.
\end{equation*}
This means that there is a finite algorithm that, for every \(d\in\mathscr{D}\) and~\(n\in\naturalswithzero\), outputs a rational number~\(q(d,n)\) that approximates the real number~\(r(d)\) with a precision of at least~\(2^{-n}\).
In particular, a real number~\(x\in\reals\) is called {\comp} if there is a recursive rational map~\(q\colon \naturalswithzero\to\rationals\) such that \(\abs{x-q(n)}<2^{-n}\) for all~\(n\in\naturalswithzero\).
An interval forecast~\(I\in\intervals\) is called {\comp} if the real numbers~\(\min I\) and~\(\max I\) are both {\comp}.
Moreover, a real process~\(\process\colon\sits\to\reals\) is called {\comp} if there is a recursive rational map~\(q\colon\sits\times\naturalswithzero\to \rationals\) such that \(\abs{\process(\sit)-q(\sit,n)}<2^{-n}\) for all~\(\sit\in\sits\) and~\(n\in\naturalswithzero\).
Finally, a forecasting system~\(\frcstsystem\in\frcstsystems\) is called {\comp} if the real processes~\(\lfrcstsystem\) and~\(\ufrcstsystem\) are both {\comp}.

So, a real process~\(\process\colon\sits\to\reals\) can be implementable by being recursive, {\lscomp}, {\uscomp} or {\comp}, but also, if it is generated by a multiplier process~\(\multprocess\), by \(\multprocess\) being of one of these four types. 
In what follows, it will be useful to have the following sets of implementable real processes at our notational disposal:
\begin{center}
\begin{tabular}{l|l}
\(\mathscr{F}_\ml\) \:&\: all {\lscomp} test processes;\\[1ex]
\(\mathscr{F}_\wml\) \:&\: all test processes generated by\\
&\: {\lscomp} multiplier processes;\\[1ex]
\(\mathscr{F}_\co=\mathscr{F}_\s\) \:&\: all recursive positive rational test processes.
\end{tabular}
\end{center}
From the discussion above, it is clear that if \(\process\) is recursive, then it is {\comp} and hence {\lscomp} as well; so \(\mathscr{F}_\s=\mathscr{F}_\co\subseteq\mathscr{F}_\ml\).
It follows from discussions elsewhere \cite[Section 5]{CoomanBock2021} that all four sets are in fact nested: \(\mathscr{F}_\s=\mathscr{F}_\co\subseteq\mathscr{F}_\wml\subseteq\mathscr{F}_\ml\).
As a result, for any \(\random\in\set{\ml,\wml,\co,\s}\), it also holds that the set~\(\mathscr{F}_\random\) is countably infinite, because the {\lscomp} test processes are countable in number; see for example \cite[Lemma 13]{Vovk2010}.

\section{Several martingale-theoretic notions of randomness}\label{sec:borrow}
Now that we know in what different ways Sceptic's betting strategies can be implementable, we are ready to introduce four different martingale-theoretic notions of randomness: Martin-Löf (\ml) randomness, weak Martin-Löf (\wml) randomness, {\comp} (\co) randomness and Schnorr (\s) randomness; except for \wml-randomness, these are all generalisations to interval forecasts of classical notions of randomness.
We recall from the Introduction that, in the martingale-theoretic setting, a path~\(\pth\in\pths\) is random for a forecasting system~\(\frcstsystem\) if Sceptic has no implementable betting strategy that is allowed by~\(\frcstsystem\) and that makes her arbitrarily rich on~\(\pth\) without borrowing.
This is true for each of the above-mentioned four notions of randomness.
The difference between them lies in the way the allowed betting strategies are implementable, and in the way Sceptic should not be able to become arbitrarily rich.
The allowed betting strategies for the notions of \ml-, \wml-, \co- and \s-randomness are gathered in the following sets:
\begin{center}
\begin{tabular}{l|l}
\(\mltests{\frcstsystem}\coloneqq\mathscr{F}_\ml\cap\testsupermartins\) \:&\: all {\lscomp} test supermartingales for~\(\frcstsystem\);\\[1ex]
\(\weakmltests{\frcstsystem}\coloneqq\mathscr{F}_\wml\cap\testsupermartins\) \:&\: all test supermartingales for~\(\frcstsystem\) generated by\\
&\: {\lscomp} multiplier processes;\\[1ex]
\(\comptests{\frcstsystem}\coloneqq\mathscr{F}_\co\cap\testsupermartins\) \:&\: all recursive positive rational test supermartingales for~\(\frcstsystem\);\\[1ex]
\(\schnorrtests{\frcstsystem}\coloneqq\mathscr{F}_\s\cap\testsupermartins\) \:&\: all recursive positive rational test supermartingales for~\(\frcstsystem\).
\end{tabular}
\end{center}
By recalling that \(\mathscr{F}_\s=\mathscr{F}_\co\subseteq\mathscr{F}_\wml\subseteq\mathscr{F}_\ml\), it readily follows that the above sets of betting strategies satisfy the following relations for any forecasting system~\(\frcstsystem\): \(\schnorrtests{\frcstsystem}=\comptests{\frcstsystem}\subseteq\weakmltests{\frcstsystem}\subseteq\mltests{\frcstsystem}\).

Now, for any~\(\random\in\set{\ml, \wml, \co}\), we will consider a path~\(\pth\in\pths\) to be \random-random for a forecasting system~\(\frcstsystem\) if Sceptic can adopt no betting strategy~\(\test\in\randomtests{\frcstsystem}\) that is \emph{unbounded} on~\(\pth\), in the sense that \(\limsup_{n\to\infty}\test(\pthton)=\infty\).

\begin{definition}[{\cite[Definition 2]{CoomanBock2021}\cite[Definition 5 and Proposition 6]{sum2020persiau}}]\label{def:notionsofrandomness}
For any~\(\random\in\set{\ml,\wml,\co}\), a path~\(\pth\in\pths\) is \random-random for a forecasting system~\(\frcstsystem\in\frcstsystems\) if no test supermartingale~\(\test\in\overline{\mathbb{T}}_\random(\frcstsystem)\) is unbounded on~\(\pth\).
\end{definition}
\noindent We want to emphasise here that, contrary to what is typically done in the literature for algorithmic randomness associated with precise forecasts, we don't necessarily impose {\compy} requirements on the forecasting systems~\(\frcstsystem\).

In the case of~\(\co\)-randomness, it may initially seem rather unintuitive that we consider recursive rational test supermartingales instead of {\comp} real non-negative supermartingales, since the naming suggests the latter.
However, as we have discussed in \cite[Section 5]{sum2020extended}, both sets of betting strategies result in the same notion of randomness; we use the former here for mathematical convenience.
In what follows, we will also implicitly use this equivalence when referring to results proven elsewhere.

To introduce \s-randomness, we require the additional notion of a \emph{real growth function} \(\realgrowth\colon\naturalswithzero\to\nonnegreals\), which is a {\comp} map from non-negative integers to non-negative reals that is non-decreasing---so \(\realgrowth(n+1)\geq\realgrowth(n)\) for all~\(n\in\naturalswithzero\)---and unbounded---so \(\lim_{n\to\infty}\realgrowth(n)=\infty\).
Now, a path~\(\pth\in\pths\) is considered to be \s-random for a forecasting system~\(\frcstsystem\) if no test supermartingale~\(\test\in\schnorrtests{\frcstsystem}\) is \emph{computably unbounded} on~\(\pth\), meaning that there is some real growth function~\(\realgrowth\) such that \(\limsup_{n\to\infty}[\test(\pthton)-\realgrowth(n)]\geq 0\).
\revised{Intuitively, and analogously to computable randomness, this means that Sceptic should not be able to adopt a recursive (positive rational) betting strategy \(\test\in\comptests{\frcstsystem}=\schnorrtests{\frcstsystem}\) that allows her to get arbitrarily rich, but now at some computable rate.}
\begin{definition}\label{def:schnorr}
A path~\(\pth\in\pths\) is \s-random for a forecasting system~\(\frcstsystem\in\frcstsystems\) if no test supermartingale~\(\test\in\schnorrtests{\frcstsystem}\) is computably unbounded on~\(\pth\).
\end{definition}

Interestingly, and similarly to \(\co\)-randomness, we can replace the set of allowable betting strategies \(\schnorrtests{\frcstsystem}\) by the set of {\comp} test supermartingales for~\(\frcstsystem\), without changing the set of~\(\s\)-random paths for~\(\frcstsystem\).
That is, a path~\(\pth\in\pths\) is \(\s\)-random for a forecasting system~\(\frcstsystem\in\frcstsystems\) if and only if there is no {\comp} betting strategy that starts with unit capital, is allowed by~\(\frcstsystem\) and makes Sceptic arbitrarily rich along~\(\pth\) at some {\comp} rate; here too, we will implicitly use this equivalence when referring to results proven elsewhere.

\begin{proposition}
A path~\(\pth\in\pths\) is \s-random for a forecasting system~\(\frcstsystem\in\frcstsystems\) if and only if no {\comp} non-negative supermartingale~\(\supermartin\in\supermartins\) is computably unbounded on~\(\pth\).
\end{proposition}

\begin{proof}
The reverse implication holds trivially since every recursive positive rational test supermartingale~\(\test \in\schnorrtests{\frcstsystem}\) is a {\comp} non-negative supermartingale for~\(\frcstsystem\) as well.

For the direct implication, we assume \emph{ex absurdo} that there is a {\comp} non-negative supermartingale~\(\supermartin\in\supermartins\) that is computably unbounded on~\(\pth\), meaning that there is a real growth function~\(\realgrowth\) such that \(\limsup_{n\to\infty}[\supermartin(\pthton)-\realgrowth(n)]\geq 0\).
By \cite[Lemma 24]{sum2020extended}, we know there is some recursive positive rational test supermartingale~\(\test\in\testsupermartins\) and a positive rational number~\(\alpha>0\) such that \(\abs{\alpha\test(\sit)-\supermartin(\sit)}\leq 7\) for all~\(\sit\in\sits\), and therefore also \(\nicefrac{(\supermartin(\sit)-7)}{\alpha}\leq\test(\sit)\) for all~\(\sit\in\sits\).
If we introduce the real growth function~\(\tilde{\realgrowth}\) defined by~\(\tilde{\realgrowth}(n)\coloneqq\max\set{\nicefrac{(\realgrowth(n)-7)}{\alpha},0}\) for all~\(n\in\naturalswithzero\), then it readily follows from \(\limsup_{n\to\infty}[\supermartin(\pthton)-\realgrowth(n)]\geq 0\) that \(\limsup_{n\to\infty}[\test(\pthton)-\tilde{\realgrowth}(n)]\geq 0\).
Hence, \(\test\in\schnorrtests{\frcstsystem}\) is computably unbounded on~\(\pth\), contradicting the assumption that \(\pth\) is \(\s\)-random for~\(\frcstsystem\).
\end{proof}

If the forecasting system~\(\frcstsystem\) is stationary in any of the above randomness notions \(\random\), that is, if \(\frcstsystem(\sit)=I\) for all~\(\sit\in\sits\), then we will simply say that a path~\(\pth\in\pths\) is \random-random for the interval forecast~\(I\), instead of saying that it is \random-random for the stationary forecasting system~\(\frcstsystem\).

We refer the reader to~\cite{CoomanBock2017,CoomanBock2021} for more information about these imprecise-probabilistic notions of randomness, and here mention only those results that are relevant to our present purposes.
Let us first mention that Definitions~\ref{def:notionsofrandomness} and~\ref{def:schnorr} are meaningful, in the sense that every forecasting system~\(\frcstsystem\in\frcstsystems\) has at least one \random-random path, with~\(\random\in\set{\ml,\wml,\co,\s}\).

\begin{proposition}[{\cite[Corollary 20]{CoomanBock2021}}]\label{prop:existence}
For any~\(\random\in\set{\ml,\wml,\co,\s}\) and any forecasting system~\(\frcstsystem\in\frcstsystems\), there is at least one path~\(\pth\in\pths\) that is \random-random for~\(\frcstsystem\).
\end{proposition}
\revised{
Conversely, there is for every path $\pth\in\pths$ at least one forecasting system $\frcstsystem\in\frcstsystems$ for which it is \random-random, with $\random\in\set{\ml,\wml,\co,\s}$. The reason is that all paths are $\random$-random with respect to the vacuous forecasting system $\sqgroup{0,1}$. To understand why this perhaps surprising result holds, it suffices to realise that the supermartingales that correspond to $\sqgroup{0,1}$ can never increase. These betting strategies therefore do not allow Skeptic to increase her capital, let alone become arbitrarily rich.}
\begin{proposition}[{\cite[Proposition 25]{CoomanBock2021}}] \label{prop:vacuous}
Consider any~\(\random\in\set{\ml,\wml,\co,\s}\).
Then any path \(\pth\in\pths\) is \random-random for the vacuous forecasting system \(\sqgroup{0,1}\).
\end{proposition}

Moreover, any path~\(\pth\in\pths\) that is \random-random for~\(\frcstsystem\in\frcstsystems\) is also \random-random for any forecasting system~that is less informative---or more conservative---than~\(\frcstsystem\).

\begin{proposition}[{\cite[Propositions 10 and 18]{CoomanBock2021}}]\label{prop:nestedfrcstsystems} 
Consider any~\(\random\in\set{\ml,\wml,\co,\s}\) and any two forecasting systems \(\frcstsystem,\frcstsystem^*\in\frcstsystems\) such that \(\frcstsystem\subseteq\frcstsystem^*\).
Then any path~\(\pth\in\pths\) that is \random-random for~\(\frcstsystem\) is also \random-random for~\(\frcstsystem^*\).
\end{proposition}

When we keep the forecasting system~\(\frcstsystem\) fixed, there is also an ordering on our four martingale-theoretic notions of randomness.
To describe this ordering, we introduce, for every \(\random\in\set{\ml,\wml,\co,\s}\) and every forecasting system~\(\frcstsystem\), the corresponding set of \random-random paths~\(\pths_\random(\frcstsystem)\coloneqq\cset{\pth\in\pths}{\pth\text{ is \random-random for }\frcstsystem}\).

\begin{proposition}[{\cite[Section 6]{CoomanBock2021}}]
For any forecasting system~\(\frcstsystem\), \(\pths_\ml(\frcstsystem)\subseteq\pths_\wml(\frcstsystem)\subseteq\pths_\co(\frcstsystem)\subseteq\pths_\s(\frcstsystem)\).
\end{proposition}

Thus, if a path~\(\pth\in\pths\) is \ml-random for a forecasting system~\(\frcstsystem\), then it is also \wml-, \co- and \s-random for~\(\frcstsystem\).
Consequently, for any given forecasting system~\(\frcstsystem\), it is for example easier for a path~\(\pth\in\pths\) to be \s-random than for it to be \ml-random.
This makes us say that \ml-randomness is a \emph{stronger} notion of randomness than \wml-, \co- and \s-randomness.
Conversely, we say that that \s-randomness is a \emph{weaker} notion of randomness than \co-, \wml- and \ml-randomness.

\section{A remarkable equivalence: non-stationary precise forecasting systems vs interval forecasts}\label{sec:difference}
For didactic reasons, we will now put a temporal halt to our rather technical but necessary introduction of mathematical concepts.
After all, now that we have these four martingale-theoretic notions of randomness and some of their properties at our disposal, we can start to address our central question in this paper: How does allowing for imprecision change our understanding of random sequences?
Is it for example `easier' to capture the randomness of some paths by imprecise forecasting systems? 
Are there paths whose randomness can only be described by imprecise uncertainty models?

To start the discussion, we fix any~\(\random\in\set{\ml,\wml,\co,\s}\) and recall from Proposition~\ref{prop:nestedfrcstsystems} that if a path~\(\pth\in\pths\) is \random-random for some precise forecasting system~\(\frcstsystem\in\frcstsystems\), then it is also \random-random for any forecasting system that is less informative.
Hence, in particular, for any real numbers~\(p,q\in\sqgroup{0,1}\) such that \(p<q\), if \(\pth\) is \random-random for the forecasting system~\(\frcstsystem_{p,q}\), defined by
\begin{align*}
\frcstsystem_{p,q}(\sit)\coloneqq\begin{cases}
p&\text{if \(\abs{\sit}\) is odd}\\
q&\text{if \(\abs{\sit}\) is even}
\end{cases}
\text{ for all }\sit\in\sits,
\end{align*}
then it is also \random-random for the interval forecast~\(\intervalpq\).
So we see that \(\intervalpq\) can be used as a simpler---because stationary---yet imprecise alternative for~\(\frcstsystem_{p,q}\).
In many cases, this procedure of replacing a precise non-stationary forecasting system by a stationary imprecise one will result in a larger set of \random-random paths, and therefore lead to a less informative description of the \(\random\)-randomness associated with~\(\pth\).
However, interval forecasts do not merely serve as an alternative for non-stationary precise forecasts.
Indeed, as was shown by De Cooman and De Bock \cite[Section 10]{CoomanBock2021}, there are paths~\(\pth\) that are \random-random for~\(\intervalpq\), but not \random-random for any (more) precise (possibly non-stationary) {\comp} forecasting system~\(\frcstsystem\).

\begin{theorem}[{\cite[Theorem 37]{CoomanBock2021}}]\label{the:R:inherently:imprecise}
Consider any~\(\random\in\set{\ml,\wml,\co,\s}\) and any interval forecast~\(\intervalpq\in\intervals\).
Then there is a path~\(\pth\in\pths\) that is \random-random for the interval forecast~\(\intervalpq\), but that is never \random-random for any {\comp} forecasting system~\(\frcstsystem\in\frcstsystems\) whose highest imprecision is smaller than that of~\(\intervalpq\), in the specific sense that \(\sup_{\sit\in\sits}[\ufrcstsystem(\sit)-\lfrcstsystem(\sit)]<q-p\).
\end{theorem}

Theorem~\ref{the:R:inherently:imprecise} led them to claim that \(\random\)-randomness is inherently imprecise, because the randomness of the paths~\(\pth\) in Theorem~\ref{the:R:inherently:imprecise} can only be captured by an imprecise forecasting system.
The following corollary, which is essentially a less technical formulation of our main result---Theorem~\ref{theorem:whatsthedifference} further on in Section~\ref{sec:proof}---shows that the assumption that \(\frcstsystem\) is {\comp} is crucial for this claim: indeed, Corollary~\ref{cor:textie} shows that there is a precise forecasting system~\(\frcstsystem\)---so with \(\sup_{\sit\in\sits}\sqgroup{\ufrcstsystem(\sit)-\lfrcstsystem(\sit)}=0\)---such that \(\pth\) is \random-random for~\(\intervalpq\) if and only if it is \random-random for~\(\frcstsystem\).
Hence, for this particular \(\frcstsystem\), there is no path~\(\pth\in\pths\) that is \random-random for~\(\sqgroup{p,q}\) but not for~\(\frcstsystem\).
We postpone an exact formulation and proof of Theorem~\ref{theorem:whatsthedifference} to Section~\ref{sec:proof}, since that requires even more mathematical technicalities.

\begin{corollary}\label{cor:textie}
Consider any~\(\random\in\set{\ml,\wml,\co,\s}\) and any interval forecast~\(\intervalpq\in\intervals\) with~\(p<q\).
Then there is a precise forecasting system~\(\frcstsystem\in\frcstsystems\), with \(\frcstsystem\subseteq \sqgroup{p,q}\), such that any path~\(\pth\in\pths\) is \random-random for~\(\frcstsystem\) if and only if it is \random-random for~\(\intervalpq\).
\end{corollary}

\begin{proof}
It is an immediate consequence of Theorem~\ref{theorem:whatsthedifference} that there is some precise forecasting system~\(\frcstsystem\in\frcstsystems\), with \(\frcstsystem\subseteq\sqgroup{p,q}\), such that a path~\(\pth\in\pths\) is \random-random for~\(\intervalpq\) if and only if it is \random-random for~\(\frcstsystem\).
\end{proof}

By Theorem~\ref{the:R:inherently:imprecise}, the precise forecasting system~\(\frcstsystem\) in Corollary~\ref{cor:textie} is then necessarily non-{\comp}, as well as non-stationary.

\begin{corollary}\label{cor:textie2}
Consider any~\(\random\in\set{\ml,\wml,\co,\s}\), any precise forecasting system~\(\frcstsystem\in\frcstsystems\) and any interval forecast~\(\sqgroup{p,q}\in\intervals\) with~\(p<q\).
If~\(\pths_\random(\frcstsystem)=\pths_\random(\sqgroup{p,q})\), then~\(\frcstsystem\) must be non-{\comp} and non-stationary.
\end{corollary}

\begin{proof}
Since \(q-p>0\) by assumption, and since \(\sup_{\sit\in\sits}[\ufrcstsystem(\sit)-\lfrcstsystem(\sit)]=0<q-p\) due to the precision of~\(\frcstsystem\), it follows from Theorem~\ref{the:R:inherently:imprecise} that \(\frcstsystem\) must be  non-{\comp}, because, otherwise, this theorem would guarantee that there is some some path~\(\pth\in\pths\) that is \random-random for~\(\intervalpq\) but not \random-random for~\(\frcstsystem\).

To conclude, we prove that the precise forecasting system~\(\frcstsystem\) can't be stationary either.
Indeed, assume \emph{ex absurdo} that \(\frcstsystem(\sit)=r\in\sqgroup{0,1}\) for all~\(\sit\in\sits\).
Theorem~\ref{the:R:inherently:imprecise} guarantees that there is some path~\(\pth\in\pths\) that is \random-random for~\(\intervalpq\), but that is not \random-random for any {\comp} forecasting system whose highest imprecision is smaller than that of~\(\intervalpq\).
By our assumption, however, \(\pth\) must also be \random-random for~\(\frcstsystem\), or in other words, for~\(r\).
Now, there always is some {\comp} interval forecast~\(\sqgroup{\underline{r},\overline{r}}\) such that \(r\in\sqgroup{\underline{r},\overline{r}}\) and~\(\overline{r}-\underline{r}<q-p\).
Since \(\pth\) is \random-random for~\(r\) by assumption, it is also \random-random for~\(\sqgroup{\underline{r},\overline{r}}\) by Proposition~\ref{prop:nestedfrcstsystems}, and this contradicts the assumption that \(\pth\) is not \random-random for any {\comp} forecasting system whose highest imprecision is smaller than that of~\(\intervalpq\).
\end{proof}
\revised{
We repeat that, in the martingale-theoretic setting, our main result complements Theorem~\ref{the:R:inherently:imprecise} by showing the importance of the computability assumption on the precise forecasting systems.
It turns out that our main result is also interesting from a measure-theoretic randomness perspective, since we'll see below it readily leads to corollaries that are reminiscent of existing measure-theoretic randomness results.
The measure-theoretic randomness notion that we consider for this purpose is \emph{uniform randomness}, which, as mentioned in the Introduction, allows us to consider the randomness of a path with respect to an effectively closed---or compact---class of measures. 
In particular, there is a test such that a path \(\pth\) passes this test if and only if it is uniformly random with respect to at least one member of the considered class of measures \cite[Theorem 5.23 and Remark 5.24]{Bienvenu2011}.
Our next result shows that our martingale-theoretic notion of randomness satisfies a similar property: a path \(\pth\in\pths\) is random for a stationary interval forecast if and only if it is random for at least one compatible precise forecasting system.

\begin{corollary}\label{cor:compatible}
Consider any~\(\random\in\set{\ml,\wml,\co,\s}\) and any stationary interval forecast~\(\intervalpq\in\intervals\).
Then a path \(\pth\in\pths\) is \random-random for \(\intervalpq\) if and only if it is \random-random for at least one compatible precise forecasting system \(\frcstsystem\in\frcstsystems\).
\end{corollary}
\begin{proof}
By Proposition~\ref{prop:nestedfrcstsystems}, the `if' part is straightforward, so we proceed to the `only if' part.
Assume that \(\pth\) is \random-random for the interval forecast \(\intervalpq\). If $p=q$, then $\pth$ is clearly \random-random for the compatible price forecasting system $p=[p,q]$. If $p<q$, then it follows from Corollary~\ref{cor:textie} that there is some precise compatible forecasting system \(\frcstsystem\in\frcstsystems\) for which \(\pth\) is \random-random.
\end{proof}

Another measure-theoretic result that we can now show has a martingale-theoretic counterpart, is the existence of a so-called \emph{neutral} measure for which all paths \(\pth\in\pths\) are random. In a measure-theoretic context, this is true for uniform randomness \cite[Theorem 6.2]{Bienvenu2011}. We here obtain a similar result for any of the four martingale-theoretic notions of randomness that we consider.

\begin{corollary}\label{cor:neutral}
Consider any \(\random\in\set{\ml,\wml,\co,\s}\).
Then there is a precise---but necessarily non-stationary and non-{\comp}---forecasting system \(\frcstsystem\in\frcstsystems\) for which all paths \(\pth\in\pths\) are \random-random.
\end{corollary}

\begin{proof}
From Corollary~\ref{cor:textie}, it immediately follows that there is some precise forecasting system \(\frcstsystem\in\frcstsystems\) such that \(\pths_\random(\frcstsystem)=\pths_\random(\sqgroup{0,1})\), and hence, by Proposition~\ref{prop:vacuous}, \(\pths_\random(\frcstsystem)=\pths\).
Corollary~\ref{cor:textie2} then implies that~\(\frcstsystem\) is necessarily non-stationary and non-{\comp}.
\end{proof}

We find this result to be particularly intriguing.
Proposition~\ref{prop:vacuous} guarantees that every path~\(\pth\in\pths\) is random for the vacuous forecasting system.
But since all precise forecasting systems are compatible with the vacuous forecasting system, Corollary~\ref{cor:compatible} then tells us that this amounts to every path~\(\pth\in\pths\) being random for at least one precise forecasting system.
The result above strengthens this, by showing that there is in fact one single precise forecasting system for wich all paths are random. 

Where does this discussion leave us?
Corollaries~\ref{cor:textie} and~\ref{cor:textie2} show that the randomness of a path~\(\pth\in\pths\) with respect to an interval forecast~\(\intervalpq\in\intervals\) with \(p<q\), be it {\comp} or not, can be equivalently described by a precise forecasting system that is then necessarily \emph{non-{\comp}} and non-stationary. This furthermore implied, as we have seen in Corollary~\ref{cor:compatible}, that randomness with respect to a stationary interval forecast can be equivalently described in terms of the compatible precise forecasting systems.
It may therefore seem that, on purely theoretical grounds, stationary imprecise forecasting systems are not needed in the study of algorithmic randomness.
However, if we want to maintain our claim that randomness is inherently imprecise, Corollaries~\ref{cor:textie2} and~\ref{cor:neutral} tell us we need only explain why we believe that non-{\comp} forecasting systems are non-satisfactory, and even fairly useless.
We will come to that in Section~\ref{sec:stat}, where we argue why the computability assumption on the forecasting system is justified on practical grounds.
But before getting to that, we now devote ourselves to the formulation and proof of Theorem~\ref{theorem:whatsthedifference}, and to introducing the requisite mathematical machinery.}

\section{Global uncertainty models and almost sure events}\label{sec:almostsure}
We will not only consider `local' gambles on the sample space~\(\set{0,1}\), as we did in Section~\ref{sec:single}, but also `global' gambles~\(\gambleglobal\colon\pths\to\reals\), which are bounded maps from the set~\(\pths\) of all paths to the real numbers.
We collect all such gambles in the set~\(\gamblesglobal\).
With every subset~\(\event\subseteq\pths\), which we call an \emph{event}, we associate the gamble~\(\ind{\event}\in\gamblesglobal\) which assumes the value~\(1\) on~\(\event\) and~\(0\) elsewhere, and call it the \emph{indicator} of~\(\event\); observe that, since \(0\leq\ind{\event}(\pth)\leq1\) for any~\(\pth\in\pths\), \(\ind{\event}\) is bounded and hence indeed a gamble.
The complement~\(\pths\setminus\event\) of an event~\(\event\in\pths\) is denoted by~\(\event^c\) and its indicator \(\ind{\event^c}\) satisfies \(\ind{\event}=1-\ind{\event^c}\).

Similarly to considering upper and lower expectations of local gambles with respect to interval forecasts in Section~\ref{sec:single}, we can also associate global upper and lower expectations with global gambles~\(\gamble\in\gamblesglobal\), but now with respect to forecasting systems \(\frcstsystem\in\frcstsystems\).
We do so by adopting the so-called game-theoretic \cite[Equations (6) and (7)]{CoomanBock2021}\footnote{Several versions of these definitions exist, which differ only in the type of supermartingales that are used (real-valued, extended real-valued, unbounded, bounded, bounded below) \cite{CoomanBock2021,tjoens2021,cooman2015:markovergodic,ShaferVovk2019,tjoens2021second}. For gambles, however, all these definitions are equivalent. We adopt the version in~\cite{CoomanBock2021} for reasons of simplicity, as it allows us to use the same supermartingales we introduced in Section~\ref{sec:multiple}.} upper expectation~\(\uexglobal\) and lower expectation~\(\lexglobal\), which are defined by
\begin{align*}
\uexglobal(\gambleglobal)
&\coloneqq
\inf\cset[\big]{\supermartin(\init)}
{\supermartin\in\supermartins[\frcstsystem]
\text{ and }
\liminf_{n\to\infty}\supermartin(\pthton)\geq \gambleglobal(\pth)
\text{ for all~\(\pth\in\pths\)}}\text{ for all } \gambleglobal\in\gamblesglobal,
\shortintertext{and}
\lexglobal(\gambleglobal)
&\coloneqq
\sup\cset[\big]{\submartin(\init)}
{\submartin\in\submartins[\frcstsystem]
\text{ and }
\limsup_{n\to\infty}\submartin(\pthton)\leq\gambleglobal(\pth)
\text{ for all~\(\pth\in\pths\)}}\text{ for all } \gambleglobal\in\gamblesglobal.
\end{align*}
These global upper and lower expectations are related to each other through the following conjugacy relationship \cite[Equation (8)]{CoomanBock2021}: \(\uexglobal(\gambleglobal)=-\lexglobal(-\gambleglobal)\) for all~\(\gambleglobal\in\gamblesglobal\).

The global upper expectation~\(\uexglobal(\gambleglobal)\) is the infimum initial capital for which Sceptic can adopt a betting strategy that guarantees her ending up with a higher capital than the reward that is associated with~\(\gambleglobal\), along all paths.
Moreover, \(\uexglobal\) satisfies the following properties, which resemble \ref{axiom:coherence:bounds}--\ref{axiom:coherence:increasingness}.
\begin{proposition}[{\cite[Proposition 2]{CoomanBock2021}}]
Consider any forecasting system~\(\frcstsystem\in\frcstsystems\).
Then for all gambles~\(\gambleglobal,g\in\gamblesglobal\) and all~\(\mu\in\reals\) and~\(\lambda\in\nonnegreals\):
\begin{enumerate}[label=\upshape{E}\arabic*.,ref=\upshape{E}\arabic*,leftmargin=*,itemsep=0pt]
\item \(\inf \gambleglobal\leq\uexglobal(\gambleglobal)\leq\sup \gambleglobal\) \hfill{\upshape[boundedness]}\label{prop:globalbound}
\item \(\uexglobal(\lambda \gambleglobal)=\lambda\uexglobal(\gambleglobal)\) \hfill{\upshape[non-negative homogeneity]}
\item \(\uexglobal(\gambleglobal+g)\leq\uexglobal(\gambleglobal)+\uexglobal(g)\) \hfill{\upshape[subadditivity]}\label{prop:subadditivity}
\item \(\uexglobal(\gambleglobal+\mu)=\uexglobal(\gambleglobal)+\mu\) \hfill{\upshape[constant additivity]}\label{prop:constantadditivity}
\item if \(\gamble\leq g\) then \(\uexglobal(\gamble)\leq\uexglobal(g)\) \hfill{\upshape[monotonicity]}\label{prop:increasingness}
\end{enumerate}
\end{proposition}

For any event~\(\event\subseteq\pths\) and any forecasting system~\(\frcstsystem\), global upper and lower expectations allow us to also  define their corresponding \emph{lower and upper probabilities}: \(\lprglobal(\event)\coloneqq\lexglobal(\ind{\event})\) and~\(\uprglobal(\event)\coloneqq\uexglobal(\ind{\event})\).
We say that an event~\(\event\) is \emph{almost sure} for a forecasting system~\(\frcstsystem\) if \(\lprglobal(\event)=1\); if the forecasting system~\(\frcstsystem\) is not important, or clear from the context, we simply say that \(\event\) is almost sure.
Observe that by \ref{prop:constantadditivity} and conjugacy,
\begin{equation*}
\uprglobal(\event^c)=\uexglobal(\ind{\event^c})=\uexglobal(1-\ind{\event})=1+\uexglobal(-\ind{\event})=1-\lexglobal(\ind{\event})=1-\lprglobal(\event)\text{ for all } \event\subseteq\pths.
\end{equation*}
Hence, an event~\(\event\) is almost sure if and only if \(\uprglobal(\event^c)=0\).
This alternative characterisation is often more convenient in proofs, and we will use it implicitly.

There are two features of almost sure events that will be useful to us. The first is that they are never empty.

\begin{proposition}\label{prop:non-empty}
Any almost sure event~\(\event\subseteq\pths\) is non-empty.
\end{proposition}

\begin{proof}
Assume \emph{ex absurdo} that \(\event\) is empty.
This would imply that \(\event^c=\pths\) and therefore, since \(\event\) is almost sure, that \(\uprglobal(\pths)=0\). But it follows from property~\ref{prop:globalbound} that, actually, \(\uprglobal(\pths)=\uexglobal(1)=1\).
\end{proof}

The second feature is that countable intersections of almost sure events are still almost sure.
We start with finite intersections.

\begin{lemma}\label{lem:finiteintersection}
Consider two almost sure events \(\event,B\subseteq\pths\), then \(\event \cap B\) is almost sure as well.
\end{lemma}

\begin{proof}
Since \(\event\) and~\(B\) are almost sure events, we know that \(\uprglobal(\event^c)=0\) and~\(\uprglobal(B^c)=0\).
By invoking properties \ref{prop:globalbound}, \ref{prop:subadditivity} and \ref{prop:increasingness}, it follows that 
\begin{align*}
0 \overset{\text{\ref{prop:globalbound}}}{\leq}
\uexglobal(\ind{(\event\cap B)^c)}) 
=\uexglobal(\ind{\event^c\cup B^c})
\overset{\text{\ref{prop:increasingness}}}{\leq}\uexglobal(\ind{\event^c}+\ind{B^c})
\overset{\text{\ref{prop:subadditivity}}}{\leq}\uexglobal(\ind{\event^c})+\uexglobal(\ind{B^c})=\uprglobal(\event^c)+ \uprglobal(B^c)=0.
\end{align*}
So \(\uprglobal((\event\cap B)^c)=0\) and, therefore, \(\event\cap B\) is almost sure.
\end{proof}

By combining this result with the following lemma, we obtain the version for countable intersections.

\begin{lemma}\label{lem:continuity}
Consider any non-decreasing sequence~\((\gambleglobal_n)_{n\in\naturalswithzero}\) in~\(\gamblesglobal\) that converges pointwise to a gamble~\(\gambleglobal\in\gamblesglobal\).
Then \(\uexglobal(\gambleglobal)=\lim_{n\to\infty}\uexglobal(\gambleglobal_n)\).
\end{lemma}

\begin{proof}
From \cite[Equation (5) and Proposition 10]{cooman2015:markovergodic}, it follows that the global upper expectation~\(\uexglobal(\gambleglobal)\), with~\(\gambleglobal\in\gamblesglobal\), can be equivalently defined in terms of bounded below supermartingales.
By \cite[Proposition 36]{tjoens2021}, this equivalence continues to hold when considering extended real-valued bounded below supermartingales.
Consequently, this lemma follows from \cite[Theorem 23]{tjoens2021}.
\end{proof}

\begin{corollary}\label{cor:intersection}
For any sequence~\((\event_n)_{n\in\naturalswithzero}\) of almost sure events, \(\bigcap_{n\in\naturalswithzero}\event_n\) is almost sure.
\end{corollary}

\begin{proof}
For any~\(n\in\naturalswithzero\), we know from Lemma~\ref{lem:finiteintersection} that the event~\(\bigcap_{k=0}^n\event_k\) is almost sure and, therefore, that \(0=\uprglobal((\bigcap_{k=0}^n\event_k)^c)=\uprglobal(\bigcup_{k=0}^n\event^c_k)\).
Since the sequence~\((\ind{\bigcup_{k=0}^n\event^c_k})_{n\in\naturalswithzero}\) in~\(\gamblesglobal\) is non-decreasing and converges pointwise to the gamble~\(\ind{\bigcup_{n\in\naturalswithzero}\event^c_n}\in\gamblesglobal\), it follows from Lemma~\ref{lem:continuity} that
\begin{align*}
\uprglobal((\bigcap_{n\in\naturalswithzero}\event_n)^c)
=\uprglobal(\bigcup_{n\in\naturalswithzero}\event^c_n)
=\uexglobal(\ind{\bigcup_{n\in\naturalswithzero}\event^c_n})
=\lim_{n\to\infty}\uexglobal(\ind{\bigcup_{k=0}^n\event^c_k})
=\lim_{n\to\infty}\uprglobal(\bigcup_{k=0}^n\event^c_k)
=0,
\end{align*}
so \(\bigcap_{n\in\naturalswithzero}\event_n\) is indeed almost sure.
\end{proof}

\section{Several frequentist notions of randomness}\label{sec:freq}
Inspired by von Mises and Wald's work \cite{wald1937,ambosspies2000,Mises1919}, which we mentioned in the Introduction, we will also consider a number of very general frequentist notions of randomness, in addition to the four martingale-theoretic randomness notions mentioned in Section~\ref{sec:borrow}.
In particular, we will consider a path~\(\pth\in\pths\) to be random for a forecasting system~\(\frcstsystem\in\frcstsystems\) and a countable set of selection processes~\(\selections\) if the frequencies of ones along all infinite subsequences of~\(\pth\) selected by these selection processes are bounded by~\(\frcstsystem\), in the following sense.

\begin{definition}\label{def:Waldrandom}
Consider any countable set of selection processes~\(\selections\).
Then a path~\(\pth\in\pths\) is \(\selections\)-random for a forecasting system~\(\frcstsystem\in\frcstsystems\) if for every selection process~\(\selection\in\selections\) for which \(\lim_{n\to\infty}\sum_{k=0}^{n-1}\selection(\pthtok)=\infty\),
\begin{equation*}
\liminf_{n\to\infty}
\frac{\sum_{k=0}^{n-1}\selection(\pthtok)[\pthatkplus-\underline{\frcstsystem}(\pthtok)]}
{\sum_{k=0}^{n-1}\selection(\pthtok)}
\geq0
\text{ and }
\limsup_{n\to\infty}
\frac{\sum_{k=0}^{n-1}\selection(\pthtok)[\pthatkplus-\overline{\frcstsystem}(\pthtok)]}
{\sum_{k=0}^{n-1}\selection(\pthtok)}
\leq0.
\end{equation*}
\end{definition}
\noindent For a stationary forecasting system~\(I\in\intervals\), the conditions in this definition simplify to the perhaps more intuitive requirement that
\begin{equation*}
\min I
\leq\liminf_{n\to\infty}
\frac{\sum_{k=0}^{n-1}\selection(\pthtok)\pthatkplus}{\sum_{k=0}^{n-1}\selection(\pthtok)}
\leq\limsup_{n\to\infty}
\frac{\sum_{k=0}^{n-1}\selection(\pthtok)\pthatkplus}{\sum_{k=0}^{n-1}\selection(\pthtok)}
\leq\max I.
\end{equation*}

If we restrict our attention to the set of all recursive selection processes---which, as mentioned in the Introduction \cite{ambosspies2000,church1940}, is countable---then the above randomness notion coincides with the notion of Church (\ch) randomness that we introduced elsewhere, and similarly for recursive temporal selection processes and weak Church (\wch) randomness \cite{floris2021ecsqaru}.
\begin{definition}[{\cite[Definition 6]{floris2021ecsqaru}}]\label{def:church:randomness}
A path~\(\pth\in\pths\) is \ch-random (\wch-random) for a forecasting system~\(\frcstsystem\in\frcstsystems\) if for every recursive (recursive temporal) selection process~\(\selection\) for which \(\lim_{n\to\infty}\sum_{k=0}^{n-1}\selection(\pthtok)=\infty\),
\begin{equation*}
\liminf_{n\to\infty}
\frac{\sum_{k=0}^{n-1}\selection(\pthtok)[\pthatkplus-\underline{\frcstsystem}(\pthtok)]}
{\sum_{k=0}^{n-1}\selection(\pthtok)}
\geq0
\text{ and }
\limsup_{n\to\infty}
\frac{\sum_{k=0}^{n-1}\selection(\pthtok)[\pthatkplus-\overline{\frcstsystem}(\pthtok)]}
{\sum_{k=0}^{n-1}\selection(\pthtok)}\leq 0.\footnote{Due to the binary character of the sample space~\(\set{0,1}\), the \ch-randomness of a path~\(\pth\in\pths\) with respect to a forecasting system~\(\frcstsystem\in\frcstsystems\) can be equivalently defined in terms of bounds on gambles~\(\gamble\in\gambles\) \cite[Theorem~23 and~24]{CoomanBock2021}. There is a similar equivalence for~\(\selections\)-randomness. It will disappear for larger sample spaces, leading to weaker and stronger versions of such randomness notions there.}
\end{equation*}
\end{definition}

In order to establish that Definition~\ref{def:Waldrandom} is meaningful, and therefore Definition~\ref{def:church:randomness} as well, we proceed as in Section~\ref{sec:borrow} and prove that for every forecasting system~\(\frcstsystem\) and every countable set of selection processes~\(\selections\), there is at least one \(\selections\)-random path for~\(\frcstsystem\).
Observe that this extends the work of Wald that we mentioned in the Introduction.
We will make use of the following lemma.

\begin{lemma}[{\cite[Theorem 21]{CoomanBock2021}}]\label{lem:well:calibrated}
For any forecasting system~\(\frcstsystem\), any selection process~\(\selection\) and any gamble~\(\gamble\in\gambles\), the event
\begin{align*}
\cset[\bigg]{\pth\in\pths}
{\text{if }\lim_{n\to\infty}\sum_{k=0}^{n-1}\selection(\pthto{k})=\infty\text{ then }\liminf_{n\to\infty}\frac{\sum_{k=0}^{n-1}\selection(\pthto{k})\sqgroup[\big]{\gamble(\pthat{k+1})-\lex_{\frcstsystem(\pthto{k})}(\gamble)}}{\sum_{k=0}^{n-1}\selection(\pthto{k})}\geq0}
\end{align*}
is almost sure for~\(\frcstsystem\).
\end{lemma}

From Lemma~\ref{lem:well:calibrated} and Corollary~\ref{cor:intersection}, it will readily follow that for any forecasting system~\(\frcstsystem\) and any countable set of selection processes~\(\selections\), the corresponding set of~\(\selections\)-random paths is almost sure for~\(\frcstsystem\), and hence, by Proposition~\ref{prop:non-empty}, there will be at least one path that is \(\selections\)-random for~\(\frcstsystem\).

\begin{proposition}\label{prop:Wald:almost:sure}
Consider any forecasting system~\(\frcstsystem\) and any countable set of selection processes~\(\selections\).
Then the event~\(\cset{\pth\in\pths}{\pth\text{ is \(\selections\)-random for }\frcstsystem}\) is almost sure for~\(\frcstsystem\).
\end{proposition}

\begin{proof}
We start by associating two events \(\event^+_\selection,\event^-_\selection\subseteq\pths\) with every selection process~\(\selection \in\selections\), defined by
\begin{align*}
\event^+_\selection
&\coloneqq\cset[\bigg]{\pth\in\pths}
{\text{if }\lim_{n\to\infty}\sum_{k=0}^{n-1}\selection(\pthtok)=\infty\text{ then }\liminf_{n\to\infty}\frac{\sum_{k=0}^{n-1}\selection(\pthtok)[\pthatkplus-\underline{\frcstsystem}(\pthtok)]}{\sum_{k=0}^{n-1}\selection(\pthtok)}\geq 0}\\
\shortintertext{and}
\event^-_\selection
&\coloneqq\cset[\bigg]{\pth\in\pths}
{\text{if }\lim_{n\to\infty}\sum_{k=0}^{n-1}\selection(\pthtok)=\infty\text{ then }\limsup_{n\to\infty}\frac{\sum_{k=0}^{n-1}\selection(\pthtok)[\pthatkplus-\overline{\frcstsystem}(\pthtok)]}{\sum_{k=0}^{n-1}\selection(\pthtok)}\leq 0}.
\end{align*}
Lemma~\ref{lem:well:calibrated} with \(\gamble=X\) implies that the event~\(\event^+_\selection\) is almost sure for~\(\frcstsystem\), and similarly, Lemma~\ref{lem:well:calibrated} with \(\gamble=-X\) implies that the event~\(\event^-_\selection\) is almost sure for~\(\frcstsystem\).

Since the set of selection processes~\(\selections\) is countable, and since we associated with every selection process~\(\selection\in\selections\) two events \(\event^+_\selection\) and~\(\event^-_\selection\), the set of all these events is countable as well, and hence, by Corollary~\ref{cor:intersection}, the event~\(\bigcap_{\selection\in\selections,z\in\set{+,-}}\event^z_\selection\) is almost sure for~\(\frcstsystem\).
By Definition~\ref{def:Waldrandom}, the set~\(\bigcap_{\selection\in\selections,z\in\set{+,-}}\event^z_\selection\) is the set of paths that are \(\selections\)-random for~\(\frcstsystem\).
Hence, indeed, the event~\(\cset{\pth\in\pths}{\pth\text{ is \(\selections\)-random for }\frcstsystem}\) is almost sure for~\(\frcstsystem\).
\end{proof}

\begin{corollary}\label{cor:Wald:at:least:one}
For any forecasting system~\(\frcstsystem\) and any countable set of selection processes~\(\selections\), there is at least one path~\(\pth\in\pths\) that is \(\selections\)-random for~\(\frcstsystem\).
\end{corollary}

\begin{proof}
This is an immediate consequence of Propositions~\ref{prop:non-empty} and~\ref{prop:Wald:almost:sure}.
\end{proof}

If we restrict our attention to interval forecasts~\(I\in\intervals\) that stay away from zero and one, in the sense that~\(I\subseteq\group{0,1}\), and only consider notions of \(\selections\)-randomness that are stronger than \wch-randomness---so with \(\selections\) a superset of the recursive temporal selection processes---then we can also say something about the character of the \(\selections\)-random paths for~\(I\): they are non-recursive.

\begin{proposition}\label{prop:nonrecursive}
Consider any path~\(\pth\in\pths\) and any interval forecast~\(I\subseteq\group{0,1}\).
If~\(\pth\) is \wch-random for~\(I\), then it is non-recursive.
The same is true if \(\pth\) is \(\selections\)-random for \(I\), with \(\selections\) a countable superset of the recursive temporal selection processes, \revised{or if $\pth$ is $\random$-random for $I$, for any~\(\random\in\set{\ml,\wml,\co,\s}\).}
\end{proposition}

\begin{proof}
Assume \emph{ex absurdo} that \(\pth\) is recursive.
Then the temporal selection processes \(\selection_0,\selection_1~\in~\smash{\selectionsrec}\), defined by \(\selection_0(n)\coloneqq 1-\smash{\pthatnplus}\) and \(\selection_1(n)\coloneqq \smash{\pthatnplus}\) for all \(n\in\naturalswithzero\), are recursive.
Clearly, since \(\pth\) is a binary infinite sequence, it holds that \(\lim_{n \to \infty}\smash{\sum_{k=0}^{n-1}}\selection_0(\pthtok)=\infty\) or \(\lim_{n \to \infty}\smash{\sum_{k=0}^{n-1}}\selection_1(\pthtok)=\infty\), and therefore that
\begin{align*}
\lim_{n \to \infty}\frac{\sum_{k=0}^{n-1}\selection_0(\pthtok)\pthatkplus}{\sum_{k=0}^{n-1}\selection_0(\pthtok)}=0
\quad
\text{or}
\quad
\lim_{n \to \infty}\frac{\sum_{k=0}^{n-1}\selection_1(\pthtok)\pthatkplus}{\sum_{k=0}^{n-1}\selection_1(\pthtok)}=1.
\end{align*}
However, due to our assumptions about~\(I\), we also know---\revised{using Definition~\ref{def:church:randomness} if $\pth$ is \wch-random, Definition~\ref{def:Waldrandom} if $\pth$ is \(\selections\)-random, or \cite[Corollary~29]{CoomanBock2021} if $\pth$ is $\random$-random}---that every recursive temporal selection process \(\selection\) for which \smash{\(\lim_{n \to+\infty}\sum_{k=0}^{n-1} \selection(\pthtok) =+\infty\)} satisfies
\begin{align*}
0 < \min I \leq  \liminf_{n \to +\infty} \frac{\sum_{k=0}^{n-1}\selection(\pthtok) \pthatkplus}{\sum_{k=0}^{n-1}\selection(\pthtok)}
\quad
\text{and}
\quad
\limsup_{n \to +\infty} \frac{\sum_{k=0}^{n-1}\selection(\pthtok) \pthatkplus}{\sum_{k=0}^{n-1}\selection(\pthtok)}
\leq\max I < 1,
\end{align*}
a contradiction.
\end{proof}

\revised{In other words, if a path is recursive, meaning that it has a finite description, then it cannot be random---in any of the senses considered in the result---with respect to an interval forecast $I$ that stays away from both zero and one. Consider for example the path that takes the value zero at even positions and one at the others. Due to the above result, this path cannot be random with respect to an interval $[p,q]$, unless either $p=0$ or $q=1$. In fact, it is easy to see that this particular path is only random for the interval \([0,1]\).}

\revised{In the next section, we will use an \(\selections\)-random path \(\pth\in\pths\) to finally craft the special precise forecasting systems we have been constantly talking about. Due to the above result, this path will necessarily be non-recursive, and hence not describable in any finite way.
}

\section{Proof of the main result}\label{sec:proof}
At this point, we have introduced the mathematical apparatus that is necessary for proving the main result of this paper; we intend to show that for any interval forecast~\(\intervalpq\in\intervals\) and any~\(\random\in\set{\ml,\wml,\co,\s}\), there is some (non-stationary non-{\comp}) precise forecasting system~\(\frcstsystem\in\frcstsystems\) that has the exact same set of \random-random paths as \(\intervalpq\).

To this end, we consider a special type of precise forecasting system.
Fix any two real numbers~\(p,q\in\sqgroup{0,1}\) and any path~\(\varpth\in\pths\), and consider the associated temporal precise forecasting system~\(\frcstsystem^\varpth_{p,q}\in\frcstsystems\), defined by
\begin{equation}\label{eq:special:forecast:system:def}
\frcstsystem^\varpth_{p,q}(n)
\coloneqq
\begin{cases}
p&\text{if }\varpthatnplus=0\\
q&\text{if }\varpthatnplus=1
\end{cases}
\text{ for all }n\in\naturalswithzero.
\end{equation}

In our main result, Theorem~\ref{theorem:whatsthedifference} below, we will in particular use paths~\(\varpth\) that are \(\selections\)-random for an interval forecast~\(I\subseteq (0,1)\), where the countable set of selection processes~\(\selections\) is of a special type.
To make clear what such~\(\selections\) are like, we start by associating with every real process~\(\process\) and every real number~\(r\in\sqgroup{0,1}\) a temporal selection process~\(\selection^r_\process\), defined by
\begin{equation}\label{eq:selection:def}
\selection^r_\process(n)
\coloneqq
\begin{cases}
1&\text{if \(\ex_r(\adddelta\process(\sit))>0\) for some \(\sit\in\sits\) with \(\abs{\sit}=n\)}\\
0&\text{if \(\ex_r(\adddelta\process(\sit))\leq0\) for all \(\sit\in\sits\) with \(\abs{\sit}=n\)}
\end{cases}
\text{ for all \(n\in\naturalswithzero\).}
\end{equation}
We use these temporal selection processes to associate with every countable set~\(\mathscr{F}\) of real processes and every two real numbers~\(p,q\in\sqgroup{0,1}\) the clearly countable set~\(\selections^{p,q}_{\mathscr{F}}\) of temporal selection processes, defined by
\begin{equation}\label{eq:selection:sets}
\selections^{p,q}_{\mathscr{F}}
\coloneqq\cset[\big]{\selection^r_\process}{\process\in\mathscr{F}\text{ and }r\in\set{p,q}}.
\end{equation}
Since \(\selections^{p,q}_{\mathscr{F}}\) is countable, Corollary~\ref{cor:Wald:at:least:one} guarantees that there is at least one path that is \(\selections^{p,q}_\mathscr{F}\)-random for a given interval forecast~\(I\subseteq\group{0,1}\).

In this construction of the sets~\(\selections^{p,q}_{\mathscr{F}}\), the specific countable sets~\(\mathscr{F}\) of real processes that we will consider, are the sets~\(\mathscr{F}_\s\), \(\mathscr{F}_\co\), \(\mathscr{F}_\wml\) and~\(\mathscr{F}_\ml\) introduced in Section~\ref{sec:comp}.
If we recall that \(\mathscr{F}_\s=\mathscr{F}_\co\subseteq\mathscr{F}_\wml\subseteq\mathscr{F}_\ml\), Equation~\eqref{eq:selection:sets} tells us that
\begin{equation}\label{eq:selections:order}
\selections^{p,q}_{\mathscr{F}_\s}
=\selections^{p,q}_{\mathscr{F}_\co}
\subseteq\selections^{p,q}_{\mathscr{F}_\wml}
\subseteq\selections^{p,q}_{\mathscr{F}_\ml}
\text{ for all \(p,q\in\sqgroup{0,1}\).}
\end{equation}

The sets~\(\selections^{p,q}_\mathscr{F}\) might look a bit artificial, but if we restrict our attention to \emph{rational} numbers~\(p,q\in\rationals\) and to~\(\random\in\set{\co,\s}\), then as our next results shows, the corresponding~\(\selections_{\mathscr{F}_\random}^{p,q}\) are in fact the set of all recursive temporal selection processes.
Since this is exactly the set of selection processes used in Definition~\ref{def:church:randomness} to define \wch-randomness, we conclude from this that in those particular cases, for any forecasting system~\(\frcstsystem\), a path~\(\varpth\in\pths\) is \(\selections_{\mathscr{F}_\random}^{p,q}\)-random for~\(\frcstsystem\) if and only if it is \(\wch\)-random for~\(\frcstsystem\).

\begin{proposition}\label{prop:wald:church}
Consider any two rational numbers~\(p,q\in\sqgroup{0,1}\) and any~\(\random\in\set{\co,\s}\).
Then \(\selections_{\mathscr{\process}_\random}^{p,q}\) consists of all recursive temporal selection processes.
\end{proposition}

\begin{proof}
We start by proving that every selection process~\(\selection\in\selections_{\mathscr{F}_\random}^{p,q}\) is recursive and temporal.
By Equation~\eqref{eq:selection:sets}, we know that~\(\selection=\selection^r_\process\) for some~\(\process\in\mathscr{F}_\random\) and~\(r\in\set{p,q}\).
Hence, by Equation~\eqref{eq:selection:def}, \(\selection=\selection^r_\process\) is temporal.
Furthermore, since the rational process~\(\process\) is recursive, since \(p,q\in\rationals\) and since there is a finite algorithm that, for every \(k\in\naturalswithzero\), can enumerate the finite number of situations~\(\sit\in\sits\) for which \(\abs{\sit}=k\), it immediately follows that there is a finite algorithm that can check the inequalities in Equation~\eqref{eq:selection:def}, so the temporal selection process~\(\selection=\selection^r_\process\) is also recursive.

That, conversely, every recursive temporal selection process belongs to~\(\selections_{\mathscr{\process}_\random}^{p,q}\),
follows directly from our next result, Proposition~\ref{prop:helpiewie}.
\end{proof}

More generally, for any~\(\random\in\set{\ml,\wml,\co,\s}\) and any two \emph{real} numbers~\(p,q\in\sqgroup{0,1}\), the set~\(\selections_{\mathscr{\process}_\random}^{p,q}\) will actually include all recursive temporal selection processes.
Therefore, if a path~\(\varpth\in\pths\) is \(\selections_{\mathscr{F}_\random}^{p,q}\)-random for a forecasting system~\(\frcstsystem\in\frcstsystems\), it will in particular also be \wch-random for~\(\frcstsystem\).
Further on, we will restrict our attention to stationary forecasting systems~\(\frcstsystem\in\frcstsystems\) that stay away from zero and one, in the sense that~\(\frcstsystem(\sit)=I\subseteq\group{0,1}\) for all~\(\sit\in\sits\), and the \(\selections_{\mathscr{F}_\random}^{p,q}\)-random paths for~\(I\) will then necessarily be non-recursive, due to Proposition~\ref{prop:nonrecursive}.

\begin{proposition}\label{prop:helpiewie}
Consider any two real numbers~\(p,q\in\sqgroup{0,1}\) and any~\(\random\in\set{\ml,\wml,\co,\s}\).
Then every recursive temporal selection process belongs to~\(\selections_{\mathscr{\process}_\random}^{p,q}\).
\end{proposition}

\begin{proof}
Fix any recursive temporal selection process~\(\selection\) and consider the temporal rational test process~\(\process\) defined by~\(\process(n)\coloneqq1+\sum_{k=0}^{n-1}\selection(k)\) for all~\(n\in\naturalswithzero\).
Since~\(\selection\) is recursive and non-negative, it readily follows that the rational test process~\(\process\) is recursive and positive, and therefore~\(\process\in\mathscr{F}_\s\).
Since~\(\mathscr{F}_\s=\mathscr{F}_\co\subseteq\mathscr{F}_\wml\subseteq\mathscr{F}_\ml\), this implies that \(\process\in\mathscr{F}_\random\), and therefore, that \(\selections_{\mathscr{F}_\random}^{p,q}\) contains the temporal selection process~\(\selection^p_\process\).

We conclude the argument by showing that \(\selection^p_\process=\selection\).
To do so, fix any~\(n\in\naturalswithzero\).
If \(\selection(n)=1\), then for all~\(\sit\in\sits\) and~\(x\in\set{0,1}\) with~\(\abs{\sit}=n\), we have that~\(\adddelta\process(\sit)(x)=\process(n+1)-\process(n)=\selection(n)=1\), and hence, by \ref{axiom:coherence:bounds}, \(\ex_p(\adddelta\process(\sit))=1>0\).
Consequently, by Equation~\eqref{eq:selection:def}, \(\selection^p_\process(n)=1=\selection(n)\).
Otherwise, if \(\selection(n)=0\), then for all~\(\sit\in\sits\) and~\(x\in\set{0,1}\) with~\(\abs{\sit}=n\), we have that~\(\adddelta\process(\sit)(x)=\process(n+1)-\process(n)=\selection(n)=0\), and hence, by \ref{axiom:coherence:bounds}, \(\ex_p(\adddelta\process(\sit))=0\). As a consequence, by Equation~\eqref{eq:selection:def}, \(\selection^p_\process(n)=0=\selection(n)\).
\end{proof}

Let us now move on to our main result, where we use the special countable sets of selection processes~\(\selections^{p,q}_\mathscr{\process}\) and the special forecasting systems \(\frcstsystem^\varpth_{p,q}\) to reveal a surprisingly close connection between non-stationary precise forecasting systems and interval forecasts.

\begin{theorem}\label{theorem:whatsthedifference}
Consider any~\(\random\in\set{\ml,\wml,\co,\s}\), any two real numbers~\(p,q\in\sqgroup{0,1}\) such that \(p<q\), any interval forecast~\(I\subseteq\group{0,1}\), any countable set of selection processes~\(\selections\supseteq\selections^{p,q}_{\mathscr{F}_\random}\), and any path~\(\varpth\in\pths\) that is \(\selections\)-random for~\(I\).
Then a path~\(\pth\in\pths\) is \random-random for~\(\frcstsystempq\) if and only if it is \random-random for~\(\intervalpq\).
\end{theorem}

In our proof for this result, we make use of the following two lemmas to prove the implementability of a number of real processes.

\begin{lemma}\label{lem:lowersemi}
Consider any {\lscomp} real process~\(\process\) and any two natural numbers~\(N,K\in\naturals\).
Then the real process~\(\tilde{\process}\), defined by
\begin{equation*}
\tilde{\process}(\sit)\coloneqq\begin{cases}
1&\text{if }\abs{\sit}\leq N\\
\frac{1}{K}\process(\sit)&\text{if }\abs{\sit}>N
\end{cases}\text{ for all }\sit\in\sits,
\end{equation*}
is {\lscomp} as well.
\end{lemma}

\begin{proof}
Since the real process~\(\process\) is {\lscomp}, there is a recursive rational map~\(q\colon\sits\times\naturalswithzero\to\rationals\) such that \(q(\sit,n+1)\geq q(\sit,n)\) and~\(\process(\sit)=\lim_{m\to\infty}q(\sit,m)\) for all~\(\sit\in\sits\) and~\(n\in\naturalswithzero\).
Consider now the recursive rational map~\(\tilde{q}\colon\sits\times\naturalswithzero\to\rationals\) defined by
\begin{equation*}
\tilde{q}(\sit,n)\coloneqq\begin{cases}
1&\text{if }\abs{\sit}\leq N\\
\frac{1}{K}q(\sit,n)&\text{if }\abs{\sit}>N
\end{cases}
\text{ for all }\sit\in\sits\text{ and }n\in\naturalswithzero.
\end{equation*}
Then for all~\(\sit\in\sits\) and~\(n\in\naturalswithzero\),
\begin{equation*}
\tilde{q}(\sit,n+1)
=\begin{cases}
1&\text{if }\abs{\sit}\leq N\\
\frac{1}{K}q(\sit,n+1)&\text{if }\abs{\sit}>N
\end{cases}
\geq
\begin{cases}
1&\text{if }\abs{\sit}\leq N\\
\frac{1}{K}q(\sit,n)&\text{if }\abs{\sit}>N
\end{cases}
=\tilde{q}(\sit,n)
\end{equation*}
and
\begin{equation*}
\lim_{m\to\infty}\tilde{q}(\sit,m)
=
\begin{cases}
1&\text{if }\abs{\sit}\leq N\\
\lim_{m\to\infty}\frac{1}{K}q(\sit,m)&\text{if }\abs{\sit}>N
\end{cases}
=
\begin{cases}
1&\text{if }\abs{\sit}\leq N\\
\frac{1}{K}\process(\sit)&\text{if }\abs{\sit}>N
\end{cases}
=\tilde{\process}(\sit),
\end{equation*}
and therefore, \(\tilde{\process}\) is {\lscomp} as well.
\end{proof}

\begin{lemma}\label{lem:lowersemiprod}
Consider any test process~\(\process\) that is generated by a {\lscomp} multiplier process, and any two natural numbers~\(N,K\in\naturals\).
Then the test process~\(\tilde{\process}\), defined by
\begin{equation*}
\tilde{\process}(\sit)\coloneqq\begin{cases}
1&\text{if }\abs{\sit}\leq N\\
\frac{1}{K}\process(\sit)&\text{if }\abs{\sit}>N
\end{cases}\text{ for all }\sit\in\sits,
\end{equation*}
is generated by a {\lscomp} multiplier process as well.
\end{lemma}

\begin{proof}
Let~\(\multprocess\) be the {\lscomp} multiplier process that generates~\(\process\), meaning that~\(\process=\mint\).

Since~\(\multprocess\) is {\lscomp}, there is a recursive rational map~\(q\colon\sits\times\set{0,1}\times\naturalswithzero\to\rationals\) such that \(q(\sit,x,n+1)\geq q(\sit,x,n)\) and~\(\multprocess(\sit)(x)=\lim_{m\to\infty}q(\sit,x,m)\) for all~\(\sit\in\sits\), \(x\in\set{0,1}\) and~\(n\in\naturalswithzero\).
Since~\(\multprocess\) is a multiplier process, it is non-negative, and hence, we can safely assume that the recursive rational map~\(q\) is non-negative as well; otherwise, we just replace it by the recursive rational map~\(\max\set{q,0}\).
Moreover, since \(\process\) is generated by the multiplier process~\(\multprocess\), it readily follows that \(\tilde{\process}\) is generated by the multiplier process~\(\tilde{\multprocess}\) defined by
\begin{equation*}
\tilde{\multprocess}(\sit)(x)
\coloneqq
\begin{cases}
1&\text{if }\abs{\sit}<N\\
\frac{1}{K}\process(\sit x)&\text{if }\abs{\sit}=N\\
\multprocess(\sit)(x)&\text{if }\abs{\sit}>N
\end{cases}
\text{ for all } \sit\in\sits\text{ and } x\in\set{0,1}.
\end{equation*}
So it suffices to prove that~\(\tilde{\multprocess}\) is {\lscomp}.
To that end, consider the recursive rational map~\(\tilde{q}\colon\sits\times\set{0,1}\times\naturalswithzero\to\rationals\) defined by
\begin{equation*}
\tilde{q}(\sit,x,n)\coloneqq\begin{cases}
1&\text{if }\abs{\sit}<N\\
\frac{1}{K}\group[\Big]{\prod_{k=0}^{N-1}q(\xvaltok,\xvalatkplus,n)}q(\sit,x,n)&\text{if }\abs{\sit}=N\\
q(\sit,x,n)&\text{if }\abs{\sit}>N
\end{cases}
\text{ for all \(\sit=\xvaltom\in\sits\), \(x\in\set{0,1}\) and \(n\in\naturalswithzero\)}.
\end{equation*}
Then for all~\(\sit=\xvaltom\in\sits\), \(x\in\set{0,1}\) and~\(n\in\naturalswithzero\),
\begin{align*}
\tilde{q}(\sit,x,n+1)
&=\begin{cases}
1&\text{if }\abs{\sit}<N\\
\frac{1}{K}\group[\Big]{\prod_{k=0}^{N-1}q(\xvaltok,\xvalatkplus,n+1)}q(\sit,x,n+1)&\text{if }\abs{\sit}=N\\
q(\sit,x,n+1)&\text{if }\abs{\sit}>N
\end{cases}\\
&\geq
\begin{cases}
\mathrlap{1}\hphantom{\frac{1}{K}\prod_{k=0}^{N-1}q(\xvaltok,\xvalatkplus,n+1)q(\sit,x,n+1)}&\text{if }\abs{\sit}<N\\
\frac{1}{K}\group[\Big]{\prod_{k=0}^{N-1}q(\xvaltok,\xvalatkplus,n)}q(\sit,x,n)&\text{if }\abs{\sit}=N\\
q(\sit,x,n)&\text{if }\abs{\sit}>N
\end{cases}\\
&=
\tilde{q}(\sit,x,n),
\end{align*}
where the inequality holds because~\(K>0\) and~\(q\geq0\), and
\begin{align*}
\lim_{m\to\infty}\tilde{q}(\sit,x,m)
&=\begin{cases}
\mathrlap{1}\hphantom{\frac{1}{K}\prod_{k=0}^{N-1}q(\xvaltok,\xvalatkplus,n+1)q(\sit,x,n+1)}&\text{if }\abs{\sit}<N\\
\frac{1}{K}\lim_{m\to\infty}\group[\Big]{\prod_{k=0}^{N-1}q(\xvaltok,\xvalatkplus,m)}q(\sit,x,m)&\text{if }\abs{\sit}=N\\
\lim_{m\to\infty}q(\sit,x,m)&\text{if }\abs{\sit}>N
\end{cases}\\
&=\begin{cases}
\mathrlap{1}\hphantom{\frac{1}{K}\prod_{k=0}^{N-1}q(\xvaltok,\xvalatkplus,n+1)q(\sit,x,n+1)}&\text{if }\abs{\sit}<N\\
\frac{1}{K}\group[\Big]{\prod_{k=0}^{N-1}\multprocess(\xvaltok)(\xvalatkplus)}\multprocess(\sit)(x)&\text{if }\abs{\sit}=N\\
\multprocess(\sit)(x)&\text{if }\abs{\sit}>N
\end{cases}\\
&=\begin{cases}
1&\text{if }\abs{\sit}<N\\
\frac{1}{K}\process(\sit x)&\text{if }\abs{\sit}=N\\
\multprocess(\sit)(x)&\text{if }\abs{\sit}>N
\end{cases}\\
&=\tilde{\multprocess}(\sit)(x),
\end{align*}
so we see that~\(\tilde{\multprocess}\) is {\lscomp}, as needed.
\end{proof}

\begin{proof}[Proof of Theorem~\ref{theorem:whatsthedifference}]
We begin with the direct implication.
Assume that \(\pth\in\pths\) is \random-random for~\(\frcstsystempq\).
Since \(\frcstsystempq(n)\subseteq\intervalpq\) for all~\(n\in\naturalswithzero\), it follows from Proposition~\ref{prop:nestedfrcstsystems} that \(\pth\) is also \random-random for~\(\intervalpq\).

To prove the converse implication, assume that \(\pth\in\pths\) is \random-random for~\(\intervalpq\).
Taking into account Definitions~\ref{def:notionsofrandomness} and~\ref{def:schnorr}, in order to prove that \(\pth\) is \random-random for~\(\frcstsystempq\), we consider any test supermartingale~\(\test\in\smash{\testsupermartinsrandomopen{\frcstsystempq}}\) and prove that it isn't unbounded on~\(\pth\) when \(\random\in\set{\ml,\wml,\co}\), and that it isn't computably unbounded on~\(\pth\) when \(\random=\s\).

To this end, consider the two temporal selection processes~\(\selection_\test^p\) and~\(\selection_\test^q\) as defined by Equation~\eqref{eq:selection:def}.
We will take a closer look at the temporal selection process~\(\selection_\test^p\) and prove that there is only a finite number of non-negative integers~\(n\in\naturalswithzero\) for which \(\selection_\test^p(n)=1\).
To this end, assume \emph{ex absurdo} that there is an infinite number of them, and therefore that \smash{\(\lim_{n\to\infty}\sum_{k=0}^{n-1}\selection_\test^p(k)=\infty\)}.
Consider any~\(k\in\naturalswithzero\) such that \(\selection_\test^p(k)=1\), then it follows from Equation~\eqref{eq:selection:def} that there is some~\(\sit\in\sits\) with~\(\abs{\sit}=k\) such that \(\ex_p(\adddelta\test(\sit))>0\).
Since \(\ex_{\frcstsystempq(\sit)}(\adddelta\test(\sit))\leq0\) (because \(\adddelta\test\) is a supermartingale for \smash{\(\frcstsystempq\)}), this implies that necessarily \(\frcstsystempq(\sit)=q\) and therefore, since \(\abs{\sit}=k\), we infer from Equation~\eqref{eq:special:forecast:system:def} that \smash{\(\varpthatkplus=1\)}.
Since this is true for every~\(k\in\naturalswithzero\) such that~\(\selection_\test^p(k)=1\), it follows that
\begin{equation}\label{eq:helpiewelpie}
\limsup_{n\to\infty} \frac{\sum_{k=0}^{n-1}\selection_\test^p(k)\varpthatkplus}{\sum_{k=0}^{n-1}\selection_\test^p(k)}
=\limsup_{n\to\infty} \frac{\sum_{k=0}^{n-1}\selection_\test^p(k)}{\sum_{k=0}^{n-1}\selection_\test^p(k)}
=1.
\end{equation}
Since \(\test\in\smash{\randomtests{\frcstsystempq}}\), and therefore also \(\test\in\mathscr{F}_\random\), we can infer from Equation~\eqref{eq:selection:sets} that \(\selection_\test^p\in\selections^{p,q}_{\mathscr{F}_\random}\subseteq\selections\).
Consequently, since \(\varpth\) is \(\selections\)-random for~\(I\) by assumption and since \smash{\(\lim_{n\to\infty} \sum_{k=0}^{n-1}\selection_\test^p(k)=\infty\)}, we find [see Definition~\ref{def:Waldrandom}] that
\begin{equation*}
\max I\geq \limsup_{n\to\infty} \frac{\sum_{k=0}^{n-1}\selection_\test^p(k)\varpth_{k+1}}{\sum_{k=0}^{n-1}\selection_\test^p(k)}\overset{\eqref{eq:helpiewelpie}}{=}1,
\end{equation*}
contradicting the assumption that \(I\subseteq\group{0,1}\).
We conclude that, indeed, there is only a finite number of non-negative integers~\(n\in\naturalswithzero\) for which \(\selection_\test^p(n)=1\).

In a completely similar manner, it can be shown that there is only a finite number of non-negative integers~\(n\in\naturalswithzero\) for which \(\selection_\test^q(n)=1\).
Indeed, assume \emph{ex absurdo} that there is an infinite number of them.
Then, by adopting a similar argument, it follows that \smash{\(\lim_{n\to\infty}\sum_{k=0}^{n-1}\selection_\test^q(k)=\infty\)}, that \smash{\(\selection_\test^q\in\selections^{p,q}_{\mathscr{F}_\random}\subseteq\selections\)}, and that for all~\(k\in\naturalswithzero\), if \(\selection_\test^q(k)=1\), then \(\varpthatkplus=0\).
That being so, it follows from Definition~\ref{def:Waldrandom}, since \(\varpth\) is \(\selections\)-random for~\(I\) by assumption, that
\begin{equation*}
\min I\leq\liminf_{n\to\infty} \frac{\sum_{k=0}^{n-1}\selection_\test^q(k)\varpth_{k+1}}{\sum_{k=0}^{n-1}\selection_\test^q(k)}=0,
\end{equation*}
contradicting the assumption that \(I\subseteq\group{0,1}\).

Since there are only a finite number of non-negative integers~\(n\in\naturalswithzero\) for which \(\selection_\test^p(n)=1\) or \(\selection_\test^q(n)=1\), and since for each such~\(n\), there is only a finite number of situations~\(\sit\in\sits\) such that~\(\abs{\sit}=n\), it follows from Equation~\eqref{eq:selection:def} that there are only a finite number of situations~\(\sit\in\sits\) for which \(\ex_p(\adddelta\test(\sit))>0\) or \(\ex_q(\adddelta\test(\sit))>0\).
Hence, there is some~\(N\in\naturals\) such that \(\uex_{\intervalpq}(\adddelta\test(\sit))=\max\set{\ex_p(\adddelta\test(\sit)),\ex_q(\adddelta\test(\sit))}\leq 0\) for all~\(\sit\in\sits\) such that \(\abs{\sit}>N\).

Let \(K\in\naturals\) be any positive natural number such that \(K>\test(\sit)\) for all~\(\sit\in\sits\) with~\(\abs{\sit}=N+1\), and consider the test process~\(\tilde{\test}\colon\sits\to\reals\) defined by
\begin{align*}
\tilde{\test}(\sit)\coloneqq\begin{cases}
1&\text{if }\abs{\sit}\leq N\\
\frac{1}{K}\test(\sit)&\text{if }\abs{\sit}>N
\end{cases}
\text{ for all }\sit\in\sits.
\end{align*}
We intend to prove that~\(\tilde{\test}\in\randomtests{\sqgroup{p,q}}\). We will do so by consecutively showing that it is a supermartingale for~\(\sqgroup{p,q}\), that it is positive if~\(\random\in\set{\co,\s}\), that it is implementable in the same way as~\(\test\) is, and that~\(\tilde{\test}(\init)=1\).

To prove that~\(\tilde{\test}\) is a supermartingale for~\(\sqgroup{p,q}\), we fix some~\(\sit\in\sits\), and consider three mutually exclusive possibilities: \(\abs{\sit}<N\), \(\abs{\sit}=N\) and \(\abs{\sit}>N\).
If \(\abs{\sit}<N\), then
\begin{equation*}
\uex_{\intervalpq}\group[\big]{\adddelta\tilde{\test}(\sit)}
=\uex_{\intervalpq}(0)
\overset{\text{\ref{axiom:coherence:bounds}}}{=}0.
\end{equation*}
If \(\abs{\sit}=N\), then
\begin{equation*}
\uex_{\intervalpq}(\adddelta\tilde{\test}(\sit))
=\uex_{\intervalpq}\group[\big]{\tilde{\test}(\sit\,\bullet)-\tilde{\test}(\sit)}
=\uex_{\intervalpq}\group[\Big]{\frac{1}{K}\test(\sit\,\bullet)-1}
\overset{\text{\ref{axiom:coherence:increasingness}}}{\leq}\uex_{\intervalpq}(1-1)
\overset{\text{\ref{axiom:coherence:bounds}}}{=}0,
\end{equation*}
where the inequality holds because \(K>\test(t)\geq0\) for all~\(t\in\sits\) with~\(\abs{t}=N+1\).
Finally, if \(\abs{\sit}>N\), then
\begin{equation*}
\uex_{\intervalpq}(\adddelta\tilde{\test}(\sit))
=\uex_{\intervalpq}\group[\Big]{\frac1K\adddelta\test(\sit)}
\overset{\text{\ref{axiom:coherence:homogeneity}}}{=}\frac1K\uex_{\intervalpq}(\adddelta\test(\sit))
\leq0,
\end{equation*}
where in the second equality and final inequality, we also used the fact that \(K>0\).

Observe that \(\tilde{\test}\) is positive if~\(\random\in\set{\co,\s}\), because then~\(\test\in\smash{\randomtests{\frcstsystempq}}\subseteq\mathscr{F}_\random\) is positive and because~\(K>0\).

Let us now prove that \(\tilde{\test}\) is implementable in the same way as \(\test\) is.
If~\(\random=\ml\), then \(\test\in\smash{\mltests{\frcstsystempq}}\subseteq\mathscr{F}_\ml\) is {\lscomp}, so it follows from Lemma~\ref{lem:lowersemi} that \(\tilde{\test}\) is {\lscomp} as well.
If~\(\random=\wml\), then \(\test\in\smash{\weakmltests{\frcstsystempq}}\subseteq\mathscr{F}_\wml\) is generated by a {\lscomp} multiplier process, so it follows from Lemma~\ref{lem:lowersemiprod} that \(\tilde{\test}\) is generated by a {\lscomp} multiplier process as well.
And finally, if~\(\random=\co\) or~\(\random=\s\), then \(\test\in\smash{\comptests{\frcstsystempq}}=\smash{\schnorrtests{\frcstsystempq}}\subseteq\mathscr{F}_\co=\mathscr{F}_\s\) is a rational and recursive process, and it is therefore obvious that this is true for~\(\tilde{\test}\) as well.

Since also \(\tilde{\test}(\init)=1\), we conclude that \(\tilde{\test}\in\testsupermartinsrandomopen{\intervalpq}\).
We now consider two possibilities.
If \(\random\in\set{\ml,\wml,\co}\), then since \(\pth\) is \random-random for~\(\intervalpq\) by assumption, \(\tilde{\test}\) can't be unbounded on~\(\pth\) by Definition~\ref{def:notionsofrandomness}.
Since also
\begin{align*}
\limsup_{n\to\infty} \tilde{\test}(\pthton)<\infty
\Rightarrow
\limsup_{n\to\infty} \frac{\test(\pthton)}{K}<\infty
\overset{K>0}{\Rightarrow}
\limsup_{n\to\infty} \test(\pthton)<\infty,
\end{align*}
it then follows that \(\test\) does not become unbounded on~\(\pth\).

If \(\random=\s\), then since \(\pth\) is \random-random for~\(\intervalpq\) by assumption, \(\tilde{\test}\) can't be computably unbounded on~\(\pth\) by Definition~\ref{def:schnorr}.
Consider now any real growth function~\(\realordering\) and an associated real growth function~\(\tilde{\realordering}\) defined by~\(\tilde{\realordering}(n)\coloneqq\nicefrac{\realordering(n)}{K}\) for all~\(n\in\naturalswithzero\).
It then holds that
\begin{equation*}
\limsup_{n\to\infty} [\tilde{\test}(\pthton)-\tilde{\realordering}(n)]<0
\Rightarrow
\limsup_{n\to\infty} \bigg[\frac{\test(\pthton)}{K}-\frac{\realordering(n)}{K}\bigg]<0
\overset{K>0}{\Rightarrow}
\limsup_{n\to\infty} \bigg[\test(\pthto{n})-\realordering(n) \bigg]<0,
\end{equation*}
and hence, since \(\tilde{\test}\) is not computably unbounded on~\(\pth\) for the real growth function~\(\tilde{\realgrowth}\), \(\test\) does not become computably unbounded on~\(\pth\) for~\(\realgrowth\).
Since this holds for any real growth function~\(\realgrowth\), we conclude that \(\test\) does not become computably unbounded on~\(\pth\).
\end{proof}

According to Theorem~\ref{theorem:whatsthedifference}, for every choice of~\(\random\) in~\(\set{\ml,\wml,\co,\s}\), there is some path~\(\varpth\in\pths\) such that the \random-random paths for the interval forecast~\(\intervalpq\) and for the temporal precise forecasting system~\(\frcstsystempq\) coincide.
Interestingly, there is also a single path~\(\varpth\in\pths\) that does this job for all four notions of randomness that we consider here.
Basically, this is true because the weaker the notion of randomness, the weaker the conditions on~\(\varpth\) that are required in Theorem~\ref{theorem:whatsthedifference}, in the sense that the minimally required countable set of selection processes~\(\selections\) becomes smaller.

\begin{corollary}
Consider any two real numbers~\(p,q\in\sqgroup{0,1}\) such that \(p<q\), any interval forecast~\(I\subseteq\group{0,1}\), any countable set of selection processes~\(\selections\supseteq\selections_{\mathscr{F}_\ml}^{p,q}\), and any path~\(\varpth\in\pths\) that is \(\selections\)-random for~\(I\).
Then, for any~\(\random\in\set{\ml,\wml,\co,\s}\), a path~\(\pth\in\pths\) is \random-random for~\(\frcstsystempq\) if and only if it is \random-random for~\(\intervalpq\).
\end{corollary}

\begin{proof}
Since \(\selections^{p,q}_{\mathscr{F}_\s}=\selections^{p,q}_{\mathscr{F}_\co}\subseteq\selections^{p,q}_{\mathscr{F}_\wml}\subseteq\selections^{p,q}_{\mathscr{F}_\ml}\) by Equation~\eqref{eq:selections:order}, this follows readily from Theorem~\ref{theorem:whatsthedifference}.
\end{proof}

If we restrict our attention to rational numbers~\(p,q\in\rationals\) and to~\(\random\in\set{\co,\s}\), then, as proved in Proposition~\ref{prop:wald:church}, the set~\(\selections_{\mathscr{F}_\random}^{p,q}\) consists of the recursive temporal selection processes, and hence, by comparing Definitions~\ref{def:Waldrandom} and~\ref{def:church:randomness}, the conditions on~\(\varpth\) that are required in Theorem~\ref{theorem:whatsthedifference} translate into~\(\varpth\) being \wch-random for an interval forecast~\(I\subseteq\group{0,1}\).

\begin{corollary}
Consider any~\(\random\in\set{\co,\s}\), any two rational numbers~\(p,q\in\sqgroup{0,1}\) such that \(p<q\), any interval forecast~\(I\subseteq\group{0,1}\), and any path~\(\varpth\in\pths\) that is \wch-random for~\(I\).
Then a path~\(\pth\in\pths\) is \random-random for~\(\frcstsystempq\) if and only if it is \random-random for~\(\intervalpq\).
\end{corollary}

\section{Theoretical and practical necessity of interval forecasts in statistics}\label{sec:stat}
Let's now zoom out and move away from the technicalities in the previous sections, in order to better understand the implications of Theorem~\ref{theorem:whatsthedifference} and its Corollary~\ref{cor:textie}.
In trying to come to a better understanding, we have found it useful to look at these results from the point of view of statistics, whose aim it is to learn an uncertainty model from data.
Regarding the data, we will consider a finite sequence~\(\pthton\) and assume that it is an initial segment of an idealised (and unobserved) path~\(\pth\) that is (\ml-, \wml-, \co- or \s-)random; there are clearly a multitude of forecasting systems for which this is the case.
Under this assumption, we will examine what forecasting systems---that make the path~\(\pth\) random---can be learned from the finite initial segment~\(\pthton\).
Notice that, whilst doing so, we have changed our point of view: instead of focusing on the paths that are random for a forecasting system~\(\frcstsystem\in\frcstsystems\), as we have done before, we have a look at the forecasting systems that make a path~\(\pth\in\pths\) random.
Even though it is commonly assumed that the uncertainty model~\(\frcstsystem\) to be estimated or identified from the data~\(\pthton\) is precise, we have put forward elsewhere \cite{CoomanBock2021} a number of arguments that question the assumption that a path's randomness should always be described by a precise forecasting system~\(\frcstsystem\).
So, in the discussion below, we want to remain open about that possibility, and see what can be said if we don't assume {\itshape a priori} that the sequence~\(\pth\) is necessarily random for a precise forecasting system.

\revised{From Proposition~\ref{prop:vacuous},} we know that there is at least one candidate (stationary) interval forecast that makes \(\pth\) random: all paths are random for the unit interval~\(\sqgroup{0,1}\).
In fact, interestingly, there is (almost always) a smallest (stationary) interval forecast~\(\intervalpq\) that makes \(\pth\) random \cite{floris2021ecsqaru}.
Meanwhile, it is not guaranteed that there is a stationary precise forecast~\(p\) that makes \(\pth\) random; the smallest (stationary) interval forecast~\(\intervalpq\) that makes \(\pth\) random needn't be a singleton \cite[Section 9.1]{CoomanBock2021}.
Hence, generally speaking, imprecision is needed if we insist on a stationary uncertainty model to describe a path's randomness.
If we also allow for non-stationary uncertainty models however, then Theorem~\ref{theorem:whatsthedifference} shows that we could replace \(\intervalpq\) by a non-stationary precise forecasting system~\(\frcstsystempq\), with~\(\varpth\) chosen as in Theorem~\ref{theorem:whatsthedifference}.
In fact, there is an even more (theoretically) straightforward way to associate a non-stationary precise forecasting system with a path~\(\pth\): the temporal forecasting system~\(\frcstsystem_{0,1}^\pth\) that assigns probability~\(1\) to the actual next value, and hence, makes a perfect prediction.

\begin{proposition}\label{prop:all}
Consider any~\(\random\in\set{\ml,\wml,\co,\s}\), then any path~\(\pth\in\pths\) is \random-random for the precise forecasting system~\(\frcstsystem_{0,1}^\pth\).
\end{proposition}

\begin{proof}
Consider any test supermartingale~\(\test\in\testsupermartinsrandomopen{\frcstsystem_{0,1}^\pth}\).
Since \(\test\) is a supermartingale for~\(\frcstsystem_{0,1}^\pth\), it holds for any~\(n\in\naturalswithzero\) that
\begin{align*}
0\geq\ex_{\frcstsystem_{0,1}^\pth(\pthton)}(\adddelta\test(\pthton))
&=\begin{cases}
\ex_0(\adddelta\test(\pthton))&\text{if }\pthatnplus=0\\
\ex_1(\adddelta\test(\pthton))&\text{if }\pthatnplus=1
\end{cases}\\
&=\begin{cases}
\mathrlap{\adddelta\test(\pthton)(0)}\hphantom{\ex_0(\adddelta\test(\pthton))}&\text{if }\pthatnplus=0\\
\adddelta\test(\pthton)(1)&\text{if }\pthatnplus=1
\end{cases}
=\adddelta\test(\pthton)(\pthatnplus),
\end{align*}
and therefore,
\begin{align*}
\test(\pthton)
=\test(\init)+\sum_{k=0}^{n-1}\adddelta\test(\pthtok)(\pthatkplus)
\leq\test(\init)=1.
\end{align*}
Consequently, all test supermartingales~\(\test\in\testsupermartinsrandomopen{\frcstsystem_{0,1}^\pth}\) are bounded above by~\(1\) along~\(\pth\).
It therefore holds [see Definitions~\ref{def:notionsofrandomness} and \ref{def:schnorr}] that \(\pth\) is \random-random for~\(\frcstsystem_{0,1}^\pth\).
\end{proof}
Hence, if \(\pth\) is random for~\(\intervalpq\), then it is also random for at least two non-stationary precise models.
We won't risk getting bogged down into a discussion on what uncertainty models are best associated with a path~\(\pth\); that would require a paper on its own.
But we do want to point out that the uncertainty models that correspond with~\(\pth\) typically do not contain the same information; that is, they do not share the same set of random paths.
Interestingly, however, as we know from Theorem~\ref{theorem:whatsthedifference}, \(\intervalpq\) and~\(\frcstsystempq\) do have the same set of random paths and are, in that sense, equally expressive.
On that ground, theoretically, one might argue that the imprecision in~\(\intervalpq\) is not needed.

We believe that this story changes when moving to more practical grounds.
If we are given an initial finite segment \(\pthton\) of a path~\(\pth\in\pths\) and want to learn a forecasting system~\(\frcstsystem\) for which \(\pth\) is random, we will have to do so by adopting a finite algorithm that, given the data \(\pthton\), outputs a forecasting system~\(\frcstsystem'\) whose set of random paths is then believed to contain~\(\pth\).
A candidate for~\(\frcstsystem'\) could be the forecasting system~\(\frcstsystem_{0,1}^\pth\) that is generated by~\(\pth\) itself.
However, it is unfeasible to learn this forecasting system, or to even approximate it, as it basically requires us to know the entire path~\(\pth\) itself.

Another candidate for~\(\frcstsystem'\) could be the non-{\comp} forecasting system~\(\frcstsystempq\).
Here too, however, it seems impossible to learn or even approximate this model because it requires us to learn the path~\(\varpi\), which is non-recursive by Proposition~\ref{prop:nonrecursive}.
At the same time, learning a stationary interval forecast~\(\intervalpq\)---which is as expressive as~\(\frcstsystempq\)---seems a much less daunting, and practically more feasible, task, especially if \(\sqgroup{p,q}\) is {\comp}.

In summary, it is one thing to associate precise uncertainty models with a path~\(\pth\) that has no precise stationary forecast, but it is another thing to actually learn them.
When it comes to the latter, {\comp} stationary interval forecasts seem more promising than non-{\comp} non-stationary precise ones.

\section{Conclusions and future work}
We conclude that precision and computability are not always compatible when describing a path's randomness.
Indeed, if you require computability, then Theorem~\ref{the:R:inherently:imprecise} shows that you should allow for imprecision as there is at least one path~\(\pth\in\pths\) whose randomness can be described by a {\comp} interval forecast~\(\sqgroup{p,q}\), but not by any {\comp} precise forecasting system~\(\frcstsystem\in\frcstsystems\).
On the other hand, if you require precision, then Theorem~\ref{the:R:inherently:imprecise} and~\ref{theorem:whatsthedifference} show that you should allow for non-computability since the path~\(\pth\) is random for the non-{\comp} precise forecasting system~\(\frcstsystempq\), but not for any {\comp} precise one.
We repeat that the above holds because interval forecasts~\(\intervalpq\) have the same sets of martingale-theoretically random paths as the related non-{\comp} non-stationary precise forecasting systems~\(\frcstsystempq\), while being simpler and stationary.
Moreover, our preliminary analysis suggests that the stationary character of interval forecasts will be of the utmost importance when moving to the field of statistics.
In particular, it seems neither possible nor opportune to try and learn---or even approximate---the non-{\comp} non-stationary precise forecasting systems~\(\frcstsystempq\), which---by definition---cannot be described by a finite algorithm, from a finite initial path segment \(\pthton\).

In our future work, we plan to further explore these preliminary ideas about a randomness-based approach to statistics, and try to develop new statistical methods based on them.

Moreover, we want to explore whether the theorems and ideas for martingale-theoretic randomness in this paper apply equally well to frequentist notions of randomness, like the ones in Section~\ref{sec:freq}.
Our preliminary investigation seems to indicate that, similarly to what Theorem~\ref{the:R:inherently:imprecise} states, there are paths that are Church random for an interval forecast~\(I\in\intervals\), but not for any {\comp} (more) precise forecasting system; it is an open question whether a similar property holds for weak Church randomness.
Meanwhile, and in contrast with Theorem~\ref{theorem:whatsthedifference}, we suspect that if there is a precise forecasting system that has the exact same set of (weak) Church random paths as a non-vanishing interval forecast, then it must be non-temporal, and therefore can't be of the form~\(\frcstsystempq\).

Lastly, we intend to explore whether the proof of Theorem~\ref{theorem:whatsthedifference} can be modified to allow for arbitrary forecasting systems \(\frcstsystem\) rather than only stationary interval forecasts.

\section{Acknowledgements}
\revised{Work on this paper was supported by the Research Foundation – Flanders (FWO), project numbers 11H5521N (for Floris Persiau) and 3G028919 (for Jasper De Bock and Gert de Cooman).}

\bibliographystyle{unsrturl} 
\bibliography{biblio}

\end{document}